\documentclass[11pt,fleqn]{article}
\usepackage{amssymb}
\usepackage{amsfonts}
\usepackage{amsmath}
\usepackage{amsthm}
\usepackage{graphicx}
\usepackage{epsfig}
\usepackage{psfrag}
\usepackage{color}
\usepackage{dsfont}
\makeatletter
\xdef\@endgadget#1{{\unskip\nobreak\hfil\penalty50\hskip1em\hbox{}\nobreak
    \hfil#1\parfillskip=0pt\finalhyphendemerits=0\par}}
\def\@qedsymbol{${}_\blacksquare$}
\def\qed{\@endgadget{\@qedsymbol}}
\newtheorem{lemma}{Lemma}[section]

\newcommand{\R}{\mathcal{R}}

\newcommand{\bq}{\begin{equation}}
\newcommand{\eq}{\end{equation}}
\newcommand{\bma}{\begin{bmatrix}}
\newcommand{\ema}{\end{bmatrix}}
\usepackage{setspace}
\usepackage{mathtools}
\usepackage{enumitem}
\usepackage{pstool}
\usepackage{times}
\usepackage[T1]{fontenc}
\usepackage{helvet}
\usepackage{marvosym}
\usepackage{cite}
\usepackage{amsmath,amssymb,amsfonts,amscd,mathrsfs}
\usepackage{algorithmic}
\usepackage{graphicx}
\usepackage{textcomp}

\newtheorem{thm}{Theorem}

\newtheorem{prn}{Proposition}
\newtheorem{cor}{Corollary}
\newtheorem{deff}{Definition}

\newtheorem{exmp}{Example}
\newtheorem{ass}{Assumption}
\newtheorem{obs}{Remark}

\def\BibTeX{{\rm B\kern-.05em{\sc i\kern-.025em b}\kern-.08em
    T\kern-.1667em\lower.7ex\hbox{E}\kern-.125emX}}

\title{\LARGE \bf Well-Posedness and Strong Convergence of Infinite-Dimensional Port-Hamiltonian Primal-Dual Gradient Dynamics}

 \author{O. V. Iftime 
\thanks{O.V. Iftime is with the Department of Economics, Econometrics and Finance, %Faculty of Economics and Business, 
and Jan C. Willems Center for Systems and Control, 
University of Groningen, Nettelbosje 2, 9747 AE, Groningen, the Netherlands. E-mail: {\tt o.v.iftime@rug.nl}} 
\;and A. J. van der Schaft 
\thanks{A.J. van der Schaft is with the Bernoulli Institute for Mathematics, Computer Science
and Artificial Intelligence, Jan C. Willems Center for Systems and Control, University of
Groningen, PO Box 407, 9700 AK, the Netherlands. E-mail: \tt{A.J.van.der.Schaft@rug.nl}}\vspace{-20pt}
}

\begin{document}

\maketitle
\thispagestyle{empty}
\pagestyle{empty}

\begin{abstract}
We establish a functional-analytic framework that justifies the LaSalle convergence argument outlined in \cite{vanderSchaft2026PDGJuly} 
on port-Hamiltonian primal-dual gradient
dynamics for finite-horizon optimal control of linear systems. In addition to strict convexity of the running cost as in \cite{vanderSchaft2026PDGJuly}, our analysis requires global Lipschitz continuity of the cost gradient. This ensures that the nonlinear gradient operator is well-defined and
Lipschitz on the entire underlying space. 
We establish global well-posedness of the optimizer
dynamics: every initial value determines a unique global mild solution, and the associated solution operators form a contraction semigroup. Similar to \cite{Gernandt2025}, this is proved using tools from maximal monotone operator theory.
Assuming existence of an optimal solution, we combine a dissipation inequality with an infinite-dimensional LaSalle invariance principle to prove strong convergence of every trajectory with a relatively compact positive orbit to the unique optimal state, control, and costate. 
Moreover, we present a sufficient condition for compactness through uniform control regularity and show that convergence holds for all initial states. Finally, we present an example with a strictly but not strongly convex running cost for which every mild solution converges strongly to the unique equilibrium, while no global exponential estimate uniformly over initial values can be obtained.
\end{abstract}


\section{Introduction}
Designing optimal controllers is often computationally intensive. A clear example is Model Predictive Control where optimal control laws need to be computed repeatedly, at a rate depending on the timescale of the to-be-controlled dynamics. This has motivated the search for \emph{suboptimal controllers}, which are less computationally intensive, but still performing close to optimal; see e.g. \cite{Rawlings2017}, \cite{Gernandt2025}, \cite{vanderSchaft2026PDGJuly} and the references quoted therein. In finite-dimensional constrained convex optimization a computationally attractive approach for computing constrained minima is the use of the \emph{primal-dual gradient algorithm}, dating back to the 1950s \cite{Arrow1952}. The core idea is to compromize between minimization of the function by gradient descent of the primal variables and constraint satisfaction by gradient ascent of the dual variables. Here the dual variables are the Lagrange multipliers associated with the constraints. It was shown in \cite{Stegink2015}, \cite{Stegink2017} that the continuous time version of the primal-dual gradient algorithm defines an incremental port-Hamiltonian system as introduced in \cite{Camlibel2013}, and further developed in \cite{Camlibel2023}. Furthermore, by using a Hamiltonian that is \emph{shifted} with respect to the minimum, the incremental port-Hamiltonian system becomes a standard port-Hamiltonian system. In particular, the shifted Hamiltonian decreases along solutions, with the running cost defining 'dissipation'. In \cite{Gernandt2025} this idea was generalized to the infinite-dimensional setting, thus covering the optimal control problem with convex running cost and linear system dynamics. 
% In particular, it was shown in \cite{Gernandt2025} that this defines an infinite-dimensional incremental port-Hamiltonian system (there called a maximally monotone port-Hamiltonian system), which can be transformed into a standard port-Hamiltonian system using a shifted Hamiltonian functional. 
A simplified formulation of the incremental port-Hamiltonian system corresponding to optimal control was recently provided in \cite{vanderSchaft2026PDGJuly}. This latter formulation remains closer to the equations provided by Pontryagin's Maximum Principle. It also also emphasizes the interpretation of the infinite-dimensional port-Hamiltonian system as a system of partial differential equations, involving 'algorithmic' time and 'physical' time. This pde formulation can be used as starting point for discretization methods in order to derive suboptimal controllers.

In the current paper we continue \cite{vanderSchaft2026PDGJuly} by providing an in-depth analysis of well-posedness of the obtained infinite-dimensional port-Hamiltonian system, and of convergence to the optimal solution (which was lacking in \cite{vanderSchaft2026PDGJuly}). 

At several instances our analysis is close to the one performed before in \cite{Gernandt2025}, especially with regard to well-posedness.
% where we utilize arguments from maximal monotone operator theory, similarly to \cite{Gernandt2025}. 
However, for convergence we follow a different path by providing a LaSalle Invariance principle proof, in the spirit of the classical paper \cite{SlemrodInfante1969}.

Specifically, the contributions of the current paper relative to
\cite{Gernandt2025} are as follows. Similarly to \cite{Gernandt2025}, we use maximal monotone operator theory to establish well-posedness of the infinite-dimensional primal-dual gradient optimizer dynamics proposed in \cite{vanderSchaft2026PDGJuly}. Compared to \cite{Gernandt2025} we allow for general nonseparable running costs $K(x,u)$, thus extending the cost formulation in \cite[Eq. (3.2)]{Gernandt2025}. Furthermore, we establish well-posedness and strong convergence under the assumption of \emph{strict} (not \emph{strong}) convexity, without the strong-accretivity hypothesis used in \cite{Gernandt2025}.
Example \ref{ex:logcosh_convergence} verifies the required conditions for a strictly, but not strongly, convex running cost and shows convergence for every mild initial value.    
	
A key difference with regard to \cite{Gernandt2025} is that we establish an infinite-dimensional LaSalle convergence result for the dynamics in \cite{vanderSchaft2026PDGJuly}. The dissipation inequality \eqref{eq:energy_identity_cont} and an invariance argument yield strong convergence \eqref{eq:compact_convergence_cont} of relatively compact trajectories to the unique optimal state, control, and costate. 
	Moreover, we give a sufficient condition for relative compactness of trajectories in terms of the uniform control regularity bound \eqref{eq:uniform_control_H1_cont}.
	% see Corollary \ref{cor:control_regularity_compact_cont}.
Example \ref{ex:logcosh_nonexpconv} further shows
that a global exponential estimate uniform over initial values can fail even when the underlying system is controllable. We believe that the proof of strong convergence based on the LaSalle Invariance principle has a clear value on its own; also when compared to the, less developed, proof context provided in \cite{Gernandt2025}. 
\textbf{Notation}  $\mathbb N$ and $\mathbb R$ are the sets of natural and real numbers, respectively. $i^2=-1$ and $\overline{\mathbb{R}}=[-\infty,+\infty]$. $I$ is identity. $A^\top$ denotes the transpose of a matrix $A\in\mathbb R^{n\times m}$. A matrix $A$ is positive definite when $A=A^\top>0$. We use the standard function spaces $C^1(\mathbb R^n\times\mathbb R^m ;\mathbb R)$ of continuously differentiable functions, and $C_c^\infty(0,T;\mathbb R^n)$ of smooth functions with compact support in $(0,T)$. 
For $d\ge 1$, $L^p(0,T;\mathbb R^d)$, $p=1,2$, denotes the usual Lebesgue space, with $L^p(0,T):=L^p(0,T;\mathbb R)$. 
$H^1(0,T;\mathbb R^n)$ is the space of $L^2$ functions
with weak derivatives in $L^2$, and $W^{1,\infty}(0,T;\mathbb R^n)$ denotes the essentially bounded functions with essentially bounded weak derivative. For an operator $L$, $D(L)$, $R(L)$, $\ker L$, $\rho(L)$ and $L^*$ 
denote its domain, range, kernel, resolvent and adjoint, respectively.

%\newpage
\section{Preliminaries}\label{sect_prel}

\noindent 

\begin{deff}{\cite[Definition 2.1 and p.~68]{BarbuPrecupanu2012}}
Let $\mathcal X$ be a real linear space. A function
$f:\mathcal X\to\overline{\mathbb{R}}$
is called \emph{convex} if
\begin{equation}
\label{eq: def_convex}
f\bigl(\lambda x+(1-\lambda)y\bigr)
\le
\lambda f(x)+(1-\lambda)f(y)
\end{equation}
for every $\lambda\in [0,1]$ and all $x,y\in \mathcal X$ such that the right-hand side is well defined. 
The function $f$ is called \emph{strictly convex} if \eqref{eq: def_convex} is strict for every $\lambda\in(0,1)$ and for every $x,y\in \mathcal X$ such that $x\neq y$, $f(x)<+\infty$ and 
$f(y)<+\infty$. 
A convex function $f:\mathcal X\to\overline{\mathbb{R}}$ is called \emph{proper} if $f(x)>-\infty$ for every $x\in \mathcal X$ and $f\not\equiv+\infty$. The \emph{effective domain} of $f$ is defined by
$\operatorname{Dom}(f)=\{x\in \mathcal X:\,f(x)<+\infty\}$.  
\end{deff}
%If $f$ is a convex function then $\operatorname{Dom}(f)$ is convex.
When $\mathcal E$ is a Hilbert space, $f$ is \emph{strongly convex} if
$f-\frac{\mu}{2}\|\cdot\|_{\mathcal E}^2$
is convex for some $\mu>0$. 

In the next subsection we present the incremental \emph{port-Hamiltonian formulation} of the primal-dual gradient dynamics for the finite-horizon convex optimal control problem, and subsequently its standard port-Hamiltonian form using a shifted Hamiltonian functional. 

\subsection{Port-Hamiltonian System Formulation} %(\cite{vanderSchaft2026PDG})}
\label{sect_IPHsys} 

Let \(T\in(0,\infty)\), \(A\in\mathbb R^{n\times n}\),
\(B\in\mathbb R^{n\times m}\), and \(x_0\in\mathbb R^n\). Define 
\[
\mathcal X_{x_0}
:=
\left\{
x\in H^1(0,T;\mathbb R^n):
x(0)=x_0
\right\},
\;
\mathcal U
:=
L^2(0,T;\mathbb R^m),
\]
and the admissible set
\[
\mathcal F
:=
\left\{
(x,u)\in\mathcal X_{x_0}\times\mathcal U:
\dot x=Ax+Bu
\quad\text{a.e. on }(0,T)
\right\}.
\]
Consider the \textit{optimal control problem} %of minimizing the cost 
\begin{equation}
	\label{eq:ocp_cost}
	\min_{(x,u)\in\mathcal F}
	\int_0^T K\bigl(x(t),u(t)\bigr)\,dt .
\end{equation}
The cost functional 
$\mathcal J(x,u):=
\int_0^T K\bigl(x(t),u(t)\bigr)\,dt$ 
is understood as an extended real-valued functional $\mathcal J:\mathcal F\to\overline{\mathbb R}$.
Suppose that the following assumptions are satisfied.
\begin{ass}\label{ass:K_in_L1}
	$K\bigl(x(\cdot),u(\cdot)\bigr)\in L^1(0,T)$ for all 
	$(x,u)\in\mathcal X_{x_0}\times\mathcal U$.
\end{ass}
\begin{ass}\label{ass: K C1}
  $K\in C^1(\mathbb R^n\times\mathbb R^m ;\mathbb R)$ is strictly convex.    
\end{ass}
\begin{ass}\label{ass:existence_optimizer}
Problem \eqref{eq:ocp_cost} admits an optimal solution
$
(x_\star,u_\star)
\in \mathcal F
%H^1(0,T;\mathbb R^n)
%\times
%L^2(0,T;\mathbb R^m)
$.  
\end{ass}
The cost functional $\mathcal J$ is finite-valued when Assumption \ref{ass:K_in_L1} is satisfied.
Since $\mathcal F$ is affine, and hence convex, and $K$ is strictly convex, the functional $\mathcal J$ is strictly convex on $\mathcal F$. 
Consequently, the optimal
pair $(x_\star,u_\star)$ is unique, up to equality almost everywhere. 

Introduce the costate $p(t)\in\mathbb R^n$, $t\in [0,T]$. Define the \textit{optimal control Hamiltonian} by
\begin{equation}
	\label{eq:Hopt}
	H_{\rm opt}(x,u,p):=
	p^\top(Ax+Bu)+K(x,u).
\end{equation}
Under regularity conditions ensuring the application of the Pontryagin minimum principle at $(x_\star,u_\star)$ (e.g.,  \cite{MeinsmaSchaft2023}\footnote{Under the assumptions of
\cite[Theorem 2.5.1, pp. 54-55]{MeinsmaSchaft2023},
and assuming in particular that the optimal control $u_\star$ is
piecewise continuous, there exists a unique costate satisfying the
Pontryagin optimality system. Moreover, the proof of \cite[Theorem~2.5.1]{MeinsmaSchaft2023} shows
that the adjoint equation has piecewise-continuous coefficients. Hence,
in the present linear setting,
$ 
p_\star\in W^{1,\infty}(0,T;\mathbb R^n)
\subset H^1(0,T;\mathbb R^n).
$}) 
there exists a costate $p_\star\in H^1(0,T;\mathbb R^n)$ s.t. 
\begin{equation}
	\label{eq:PMP}
	\begin{aligned}
		\dot x_\star(t)
		&=
		Ax_\star(t)+Bu_\star(t),
		\qquad
		x_\star(0)=x_0, \\
		0
		&=
		B^\top p_\star(t)+K_u(x_\star(t),u_\star(t)),
		\\
		\dot p_\star(t)
		&=
		-A^\top p_\star(t)-K_x(x_\star(t),u_\star(t)),
		\qquad
		p_\star(T)=0, 
	\end{aligned}
\end{equation}
where $K_x$ and $K_u$ denote the partial gradients of $K$ with
respect to its first and second arguments, respectively, and the differential equations and the stationarity condition hold
a.e. on $(0,T)$. The terminal condition $p_\star(T)=0$ results from the absence of a terminal cost and from the fact that $x_\star(T)$ is free. 
Equivalently, the system dynamics can be represented by the equality constraint 
$Ax+Bu-\dot x=0$ in $L^2(0,T;\mathbb R^n)$, or, equivalently, as a family of pointwise equality constraints $Ax(t)+Bu(t)-\dot x(t)=0$ a.e. on $(0,T)$. 
This leads to the infinite-dimensional Lagrangian
\begin{align}
\label{eq:Lagrangian_functional}
\mathcal L(x,u,p) :=
\int_0^T
& \Bigl[
p(t)^\top\bigl(Ax(t)+Bu(t)-\dot x(t)\bigr)\\
& +
K\bigl(x(t),u(t)\bigr)
\Bigr]\,dt ,\nonumber 
\end{align}
with 
$ \mathcal D(\mathcal L):=
    \mathcal X_{x_0}
    \times
    L^2(0,T;\mathbb R^m)
    \times
    L^2(0,T;\mathbb R^n)$. 
For every
$(x,u,p)\in\mathcal D(\mathcal L)$, one has
$Ax+Bu-\dot x\in L^2(0,T;\mathbb R^n)$, and H\"older's inequality gives $p^\top(Ax+Bu-\dot{x})\in L^1(0,T)$. Together with
Assumption~\ref{ass:K_in_L1}, this ensures that
$\mathcal L$ is well defined on $\mathcal D(\mathcal L)$.

Let $x=x(t,\tau)$, $u=u(t,\tau)$ and $p=p(t,\tau)$, where $t\in[0,T]$ is the physical time and $\tau\geq 0$ is the artificial algorithmic time $\tau\geq 0$. For every $\tau\geq 0$ impose the endpoint conditions $x(0,\tau)=x_0$ and $p(T,\tau)=0$. 
The continuous time primal-dual gradient dynamics in the artificial algorithmic time $\tau\geq 0$ (along which the primal-dual gradient dynamics evolve) is formally given by
\begin{equation}
	\label{eq:PDGCT_variational}
	\frac{\partial}{\partial \tau}
	\begin{pmatrix}
		Xx\\
		Uu\\
		Pp
	\end{pmatrix}
	=
	\begin{pmatrix}
		-\delta\mathcal L/\delta x\\
		-\delta\mathcal L/\delta u\\
		\phantom{-}\delta\mathcal L/\delta p
	\end{pmatrix},
\end{equation}
where $X\in\mathbb{R}^{n\times n}$, $U\in\mathbb{R}^{m\times m}$, and $P\in\mathbb{R}^{n\times n}$ are positive definite weighting matrices acting pointwise in $t$. 
The following calculation is formal and is performed for
$x\in\mathcal X_{x_0}$, $u\in\mathcal U$, and
$p\in H^1(0,T;\mathbb R^n)$ with $p(T)=0$, assuming that
$K_x(x,u)$ and $K_u(x,u)$ belong to their respective
$L^2$ spaces. The operator realization and well-posedness
are addressed in Section III.
Computing the variational derivatives $\delta\mathcal L/\delta x$, $\delta\mathcal L/\delta u$, and $\delta\mathcal L/\delta p$,  the primal-dual gradient dynamics take the formal differential
form (as in \cite[(17)-(18)]{vanderSchaft2026PDG})
\begin{equation}
	\label{eq:PDGCT_explicit}
	\frac{\partial}{\partial\tau}
	\begin{pmatrix}
		Xx\\
		Uu\\
		Pp
	\end{pmatrix}
	=
	J
	\begin{pmatrix}
		x\\
		u\\
		p
	\end{pmatrix}
	-
	\begin{pmatrix}
		K_x(x,u)\\
		K_u(x,u)\\
		0
	\end{pmatrix},
\end{equation}
where $J$ denotes the formal differential expression 
\begin{equation}
	\label{eq:J_formal}
	J
	:=
	\begin{pmatrix}
		0 & 0 & -\partial_t-A^\top\\
		0 & 0 & -B^\top\\
		-\partial_t+A & B & 0
	\end{pmatrix}.
\end{equation}
%Note that \eqref{eq:PDGCT_explicit} consists of a formally lossless skew part and a monotone dissipative part induced by the convex running cost $K$.
Next, introduce the variables $z(t,\tau):=Xx(t,\tau)$, $	v(t,\tau):=Uu(t,\tau)$, and $r(t,\tau):=Pp(t,\tau)$ for $t\in[0,T]$ and $\tau\geq 0$. For any fixed $\tau$, the variational derivatives are taken with respect to the functions of $t$. 
The \textit{quadratic Hamiltonian functional} is 
\begin{align}
	\label{eq:Hamiltonian_functional}
	\mathcal H(z,v,r)
	:=
	\frac12\int_0^T
	&\left[
	   z(t)^\top X^{-1}z(t)
	+
	v(t)^\top U^{-1}v(t)\right.\\
	& +
	  \left. r(t)^\top P^{-1}r(t)
	\right]dt . \nonumber
\end{align}
With respect to the standard $L^2$-inner products, its variational derivatives are $\delta\mathcal H/\delta z=X^{-1}z=x$, $\delta\mathcal H/\delta v=U^{-1}v=u$, and  $\delta\mathcal H/\delta r=P^{-1}r=p$. Then, as shown in \cite{vanderSchaft2026PDG}, \eqref{eq:PDGCT_explicit} defines an infinite-dimensional incremental port-Hamiltonian system in the sense of \cite{Camlibel2013, Camlibel2023}.
%\eqref{eq:PDGCT_explicit}  has a formal \textit{infinite-%dimensional port-Hamiltonian representation}. 
Furthermore, let
$z_\star(t):=Xx_\star(t)$, 
$v_\star(t):=Uu_\star(t)$, 
$r_\star(t):=Pp_\star(t)$, 
where $(x_\star,u_\star,p_\star)$ solves the Pontryagin system \eqref{eq:PMP}.  Define the \textit{shifted Hamiltonian functional} 
\begin{equation}
	\label{eq:shifted_Hamiltonian}
	\begin{aligned}
		\mathcal H_e(z,v,r)
		:=
		\frac12\int_0^T
		&\Big[
		(z-z_\star)^\top X^{-1}(z-z_\star)\\
		&+
		(v-v_\star)^\top U^{-1}(v-v_\star)
		\\
		&+
		(r-r_\star)^\top P^{-1}(r-r_\star)
		\Big]dt .
	\end{aligned}
\end{equation}
Since $(x_\star,u_\star,p_\star)$ is an equilibrium of \eqref{eq:PDGCT_explicit}, the shifted system is 
\begin{equation}
	\label{eq:shifted_pH_system}
	\frac{\partial}{\partial\tau}
	\begin{pmatrix}
		z\\
		v\\
		r
	\end{pmatrix}
	=
	J
	\begin{pmatrix}
		\delta\mathcal H_e/\delta z\\
		\delta\mathcal H_e/\delta v\\
		\delta\mathcal H_e/\delta r
	\end{pmatrix}
	-
	\begin{pmatrix}
		K_x(x,u)-K_x(x_\star,u_\star)\\
		K_u(x,u)-K_u(x_\star,u_\star)\\
		0
	\end{pmatrix}.
\end{equation}
Formally, \eqref{eq:shifted_pH_system} constitutes a true \textit{infinite-dimensional port-Hamiltonian system}, cf. \cite{vanderSchaft2026PDG}. It forms the starting point of our well-posedness and convergence analysis of the primal-dual gradient PDE. To  perform a rigorous analysis, we need several functional analytic results. These are presented in the next subsection. 

\vspace{-0.1cm}
\subsection{Functional Analytic Preliminaries}\label{sect_FuncAnPrel}

We recall some facts from convex analysis and maximal monotone operator theory. Throughout this section, $\mathcal X$ denotes a real Banach space with dual $\mathcal X^*$, and $\langle\cdot,\cdot\rangle$ denotes the duality pairing between $\mathcal X$ and $\mathcal X^*$. 
5
Recall the following terminology, used for instance in
\cite[Theorem 1.106, p. 36]{BarbuPrecupanu2012}. 
\begin{deff}{\cite[Definition 1.137, p. 53]{BarbuPrecupanu2012}}
Let $\mathcal X$ be a linear Banach space. A single-valued operator
$
T:\mathcal X\to \mathcal X^*
$
is \emph{demicontinuous} if
\(x_n\to x\) strongly in $\mathcal X$ implies
$Tx_n\rightharpoonup^* Tx$ weak-$*$ in $\mathcal X^*$.
\end{deff}
When $\mathcal X$ is a Hilbert space, $\mathcal X^*$ is identified with $\mathcal X$ via the Riesz representation theorem. Hence, an operator 
$T:\mathcal X\rightarrow \mathcal X$ is demicontinuous if $x_n\to x$ strongly in $X$ implies
$Tx_n\rightharpoonup Tx$
weakly in $\mathcal X$.  
Since strong convergence implies weak convergence, every strongly continuous operator is demicontinuous. Consequently, every globally Lipschitz operator is demicontinuous.

%Next, recall the notion of subdifferential. 
\begin{deff}{\cite[Definition 2.30, p.~82]{BarbuPrecupanu2012}}
\label{deff: Definition 2.30}
Let $f:\mathcal X\to(-\infty,+\infty]$ be a proper convex function.
The \emph{subdifferential} of $f$ is the (generally multivalued)
mapping
$
\partial f:\mathcal X\rightrightarrows \mathcal X^*
$
defined, for $x\in\operatorname{Dom}(f)$, by
\begin{equation}\label{eqn: Eq 2.33 Barbu}
\partial f(x)
=
\left\{
x^*\in \mathcal X^* :
f(x)-f(u)\le \langle x-u,x^*\rangle,
\quad \forall\,u\in \mathcal X
\right\}.
\end{equation}
For $x\notin\operatorname{Dom}(f)$, we set
$\partial f(x):=\varnothing$.
The elements $x^*\in\partial f(x)$ are called
\textit{subgradients} of $f$ at $x$.
\end{deff}

From \eqref{eqn: Eq 2.33 Barbu} 
%and the above convention,
it follows that $\partial f(x)$ is always a closed convex
subset of $\mathcal X^*$.
The set $\partial f(x)$ may be empty (for example, if $f(x)=+\infty$).
\begin{deff}{\cite[Definition 1.135, p. 53]{BarbuPrecupanu2012}}
Let $\mathcal X$ be a Banach space with dual $\mathcal X^*$. An operator (or relation) $T\subset \mathcal X\times \mathcal X^*$ is called \textit{monotone} if
$\langle x-y,x^*-y^*\rangle\geq 0$, 
for all $(x,x^*),(y,y^*)\in T$. It is called \textit{maximal monotone} if it is monotone and admits no proper monotone extension, i.e., there exists no monotone operator 
$\widetilde{T}\subset \mathcal X\times \mathcal X^*$ such that $T\subsetneq\widetilde{T}$.
\end{deff}
%
% By \cite[Theorem 2.43, pg. 88]{BarbuPrecupanu2012}, the subdifferential of a proper convex and lower-semicontinuous function on a linear Banach space is maximal monotone.
%
\begin{thm}{\cite[Theorem 2.43, pg. 88]{BarbuPrecupanu2012}}
\label{thm: Theorem 2.43}
Let $\mathcal X$ be a real linear Banach space, and let $f$ be a \emph{lower semicontinuous} and \emph{proper convex} function. Then the subdifferential operator
$\partial f:\mathcal X\rightrightarrows \mathcal X^*$ 
is a maximal monotone operator from $\mathcal X$ into $\mathcal X^*$.
\end{thm}
%
%\noindent So, Theorem \ref{thm: Theorem 2.43} gives conditions when (in our case) our subdifferential operator is maximal monotone. To prove these, we use Lemma \ref{lem: Corollary 1.140} and the  characterization for maximal monotone operators provided in Theorem \ref{thm: Theorem 1.141}.

\begin{lemma}{\cite[Corollary 1.140, p. 54]{BarbuPrecupanu2012}}
\label{lem: Corollary 1.140}
Let $\mathcal X$ be a reflexive Banach space and 
$M:\mathcal X\to \mathcal X^*$ be a \emph{monotone}, \emph{demicontinuous} and \emph{bounded} operator (maps bounded subsets of $\mathcal X$ into bounded subsets of $\mathcal X^*$, \cite[Definition 1.136, p. 53]{BarbuPrecupanu2012}) \emph{satisfying the coercivity condition}
%\[
$\displaystyle
\frac{\langle Mx,x\rangle}{\|x\|}
\rightarrow +\infty$,
as $\|x\|\to\infty$,
%\]
i.e., condition \cite[$(1.104)$, p. 54]{BarbuPrecupanu2012}.

Let $T\subset \mathcal X\times \mathcal X^*$ be \emph{a maximal monotone operator}. Then
$
R(T+M)=\mathcal X^*,
$
where $R(T+M)$ denotes the range of %the operator 
$T+M$.
\end{lemma}
\begin{thm}{\cite[Theorem 1.141, pg. 54-55]{BarbuPrecupanu2012}}
\label{thm: Theorem 1.141}
Let $\mathcal X$ be a reflexive Banach space with a norm
such that both $\mathcal X$ and $\mathcal X^*$, with the corresponding dual norm,
are strictly convex. Let $F:\mathcal X\to \mathcal X^*$ denote the normalized
duality mapping, characterized by
%\[
$\langle x,F(x)\rangle
=\|x\|_\mathcal X^2
=\|F(x)\|_{\mathcal X^*}^2$,
%\qquad 
$x\in \mathcal X$.
%\]
Let $T\subset \mathcal X\times \mathcal X^*$ be \emph{any monotone subset}. Then $T$ is maximal monotone in $\mathcal X\times \mathcal X^*$ if and only if, for every $\lambda>0$ (equivalently, for some $\lambda>0$), we have the range condition
$
R(T+\lambda F)=\mathcal X^*.
$
\end{thm}
For $\mathcal X$ real Hilbert space, %$F=I$ and 
the range condition is $R(T+\lambda I)=\mathcal X$.

%%%%%%%%%%%%%%%%%%%%%%%%%%%%%%%%%%%%%%%%%%%%%%%%%%%%%%%%%%%%%%%%%

\section{Well-posedness of the optimizer dynamics}

The operator $J$, introduced formally in \eqref{eq:J_formal} %(Subsection~\ref{sect_IPHsys}, 
acts on triples 
$(x,u,p)$ with $x,p\in H^1(0,T;\mathbb R^n)$ and
$u\in L^2(0,T;\mathbb R^m)$, satisfying the boundary conditions
$x(0)=x_0$ and $p(T)=0$. Its range is contained in the Hilbert space
\begin{equation}
\label{eqn:X_Definition}
\mathcal X:=
L^2(0,T;\mathbb R^n)
\times
L^2(0,T;\mathbb R^m)
\times
L^2(0,T;\mathbb R^n).
\end{equation}

For the rigorous well-posedness analysis, we first homogenize the
inhomogeneous initial condition by writing
%\[
$x=\widetilde x+\bar x$,
%\qquad
$\bar x(0)=x_0$,
%\]
with $\bar x \in H^1(0,T;\mathbb R^n)$ a fixed function independent of algorithmic time, so that $\widetilde x(0)=0$. Define the operator
$
J:D(J)\subset \mathcal X\to \mathcal X
$ 
by
\begin{equation}
\label{eqn:J_Definition}
J
\begin{pmatrix}
\widetilde{x}\\
u\\
p
\end{pmatrix}
=
\begin{pmatrix}
-\dot{p}-A^\top p\\
-B^\top p\\
-\dot{\widetilde{x}}+A\widetilde{x}+Bu
\end{pmatrix},
\end{equation}
with domain
\begin{equation}
\label{eqn:J_Domain}
D(J)
=
\left\{
(\widetilde{x},u,p)\in X:
\begin{array}{l}
\widetilde{x}\in H^1(0,T;\mathbb R^n),\\
u\in L^2(0,T;\mathbb R^m),\\
p\in H^1(0,T;\mathbb R^n),\\
\widetilde{x}(0)=0,\quad p(T)=0
\end{array}
\right\}.
\end{equation}
The optimizer state is regarded as an element
$\xi=(\widetilde{x},u,p)\in X$; the action $J\xi$ is defined for $\xi\in D(J)$.

\begin{prn}\label{prn:-J_max_monotone}
    The operator $-J:D(J)\subset \mathcal X\to \mathcal X$, with domain \eqref{eqn:J_Domain} and $J$ as in \eqref{eqn:J_Definition}, is skew-adjoint and maximal monotone.
\end{prn}
\begin{proof} 
A proof is provided in Appendix \ref{App: proof_Proposition1}.  
\end{proof}
Define  $\Phi : \mathcal X \rightarrow (-\infty,+\infty]$ as 
\begin{equation}
    \label{eqn: Phi_Definition}
\Phi(\widetilde x,u,p)
:=
\int_0^T
K(\widetilde x(t)+\bar x(t),u(t))
\,dt,
\end{equation}
if $K\bigl(\widetilde{x}+\bar{x},u\bigr)\in L^1(0,T)$, and set $\Phi(\widetilde x,u,p)=+\infty$ otherwise. 
Here $K:\mathbb{R}^n\times\mathbb{R}^m\to\mathbb{R}$ is the running cost (see Section \ref{sect_prel}), and
\begin{equation}
\label{eqn: Phi_Domain}
D(\Phi)=
\left\{
(\widetilde x,u,p)\in \mathcal X:
K(\widetilde x+\bar x,u)\in L^1(0,T)
\right\}.
\end{equation}
For every $(\widetilde x,u,p)\in D(J)$, one has
$(\widetilde x+\bar x,u)\in\mathcal X_{x_0}\times\mathcal U$.
Consequently, Assumption \ref{ass:K_in_L1} ensures that
$\Phi$ is finite-valued on $D(J)$, but does not give that $\Phi$ is finite-valued on whole $\mathcal X$.
Thus
$D(J)\subseteq\operatorname{dom}\Phi\subseteq \mathcal X$, and $\Phi$ does not depend on $p$ (so $p$ is arbitrary in $L^2$) so $D(J)$ may not be equal to $D(\Phi)$. 
We next state the following result. 
% show that $\Phi$ is proper, convex, and lower
% semicontinuous with respect to the norm topology of $\mathcal X$, and that its subdifferential (see Definition \ref{deff: Definition 2.30}) is maximal monotone. 
%
\begin{prn}\label{prn:Phi_max_monotone}
    $\Phi$ is proper convex and lower semicontinuous and $\partial\Phi$ is maximal monotone on $\mathcal X$. 
\end{prn}
\begin{proof}
A proof is provided in Appendix \ref{App: proof_Proposition2}.
\end{proof}
\begin{ass}\label{ass: K Global Lipschitz}
    $\nabla K$ is globally Lipschitz, i.e., there exists a constant $L_K>0$ such that
\[
\|\nabla K(x_1,u_1)-\nabla K(x_2,u_2)\|
\le
L_K\,\|(x_1,u_1)-(x_2,u_2)\|
\]
for all $(x_1,u_1),(x_2,u_2)\in\mathbb R^n\times\mathbb R^m$.
\end{ass}
\begin{obs}
The global Lipschitz continuity of $\nabla K$ ensures that
the operator
%\[
$(\widetilde x,u) \\ \longmapsto
\bigl(K_x(\widetilde x+\bar x,u),
      K_u(\widetilde x+\bar x,u)\bigr)$,
%\]
understood pointwise a.e. on $(0,T)$, maps
$L^2(0,T;\mathbb R^n)\times L^2(0,T;\mathbb R^m)$
into itself and is globally Lipschitz with constant $L_K$.
So, it is strongly continuous (and hence
demicontinuous; see
Subsection \ref{sect_FuncAnPrel}).
%These properties will be used in the subsequent perturbation arguments. 
Moreover 
%under Assumption \ref{ass: K Global Lipschitz}, $\Phi$ is finite-valued on $\mathcal X$ and 
%Indeed, global Lipschitz continuity of $\nabla K$ gives
\[
|K(x,u)|
\le
|K(0,0)|
+\|\nabla K(0,0)\|\,\|(x,u)\|
+\frac{L_K}{2}\|(x,u)\|^2.
\]
For every $(\widetilde x,u,p)\in \mathcal X$, one has
$(\widetilde x+\bar x,u)\in L^2(0,T;\mathbb R^{n+m})$.
Since $T<\infty$, the above estimate implies 
$K(\widetilde x+\bar x,u)\in L^1(0,T)$.
So, $\Phi$ is finite-valued on $\mathcal X$, consequently, $D(\Phi)=\mathcal X$. Since $\Phi$ is proper, convex, and lower semicontinuous,
\cite[Corollary 2.38, equation (2.37), p.85]{BarbuPrecupanu2012} gives
$X=\operatorname{int}D(\Phi)
\subseteq D(\partial\Phi)
\subseteq D(\Phi)=\mathcal X$. 
Consequently, under
Assumption \ref{ass: K Global Lipschitz}, we have
\begin{equation}\label{eq:DomainX}
D(\partial\Phi)=D(\Phi)=\mathcal X.  
\end{equation} 
\end{obs}
The next example shows that Assumptions \ref{ass:K_in_L1},
\ref{ass: K C1}, and \ref{ass:existence_optimizer}
alone do not provide a valid definition of the
formal optimizer dynamics on $D(J)$ (as in \eqref{eqn:J_Domain}) with values in $\mathcal X$ (as in \eqref{eqn:X_Definition}),  
following the formal framework in \cite[(17)-(19)]{vanderSchaft2026PDG}.  
This also applies to the shifted representation
\eqref{eq:shifted_pH_system} corresponding to
\cite[(23)]{vanderSchaft2026PDG}. 
%Assumption \ref{ass: K Global Lipschitz} is sufficient for well posedness. 
\begin{exmp}\label{ex:Counterexample} 
Take $T=1$, $n=m=1$, $A=-1$, $B=1$, $x_0=0$, and
$\bar x\equiv0$. 
Consider 
\begin{equation}\label{eq:explicit_smooth_counterexample_K}
\begin{aligned}
K(x,u)={}(1+x)^6+(1+x)^2(1+u^2)^{2/3}
         +u^2-8x-2
\end{aligned}
\end{equation}
%The mixed term in $K$ has at most $|u|^{4/3}$ grows for a fixed $x$. This permits integrability for bounded states and controls in $L^2$.  
For this $K$ Assumption \ref{ass: K Global Lipschitz} does not hold since $K_{xx}(x,0)=30(1+x)^4+2$ is unbounded. 
One can immediately verify that, for $K$ in \eqref{eq:explicit_smooth_counterexample_K}, Assumptions \ref{ass:K_in_L1},  \ref{ass: K C1} hold, and  Assumption  \ref{ass:existence_optimizer} is  satisfied with the optimal solution $
(x_\star,u_\star)=(0,0)
\in \mathcal F$, $p_\star=0$ and $\nabla K(0,0)=0$. We show that $K_x(x(\cdot),u(\cdot))$ may not belong to $L^2(0,1)$. Indeed, take $x^\dagger(t)=\frac{5}{3}t^{3/5}$, $u^\dagger(t)=t^{-2/5}+\frac53t^{3/5}$ and $p^\dagger(t)=0$. Then we have $x^\dagger\in H^1(0,1)$,
$x^\dagger(0)=x_0=0$, $u^\dagger\in L^2(0,1)$, and
%\[
$\dot x^\dagger=t^{-2/5}=-x^\dagger+u^\dagger$,
%\qquad 
$x^\dagger(0)=p^\dagger(1)=0$.
%\]
Note that the corresponding functional in \eqref{eqn: Phi_Definition} has its domain $D(\Phi)=L^6(0,1)\times L^2(0,1)\times L^2(0,1)$.
We have $\xi^\dagger:=(x^\dagger,u^\dagger,p^\dagger)
\in D(J)\cap D(\Phi)$, $J\xi^\dagger=0$, 
$\Phi(\xi^\dagger)<\infty$, and 
$K_u(x^\dagger,u^\dagger)\in L^2(0,1)$. 
%for every $(x,u,p)\in D(\Phi)$. 
Then, as $t$ decreases to zero,     
$\displaystyle 
\Big( K_x(x^\dagger (t),u^\dagger (t))\Big)^2 \sim 4t^{-\frac{16}{15}}$ that gives 
$K_x(x^\dagger,u^\dagger)\notin L^2(0,1)$. Recall that 
$\partial\Phi(x,u,p)
=\bigl\{(K_x(x,u),K_u(x,u),0)\bigr\}$
whenever both derivatives are in $L^2(0,1)$, and is
empty otherwise. Since $K_x(x^\dagger,u^\dagger)\notin L^2(0,1)$, then 
$\xi^\dagger\notin D(\partial\Phi)$ and 
\[
D(J)\cap D(\partial\Phi)
\subsetneq D(J)\cap D(\Phi)=D(J).
\]
\end{exmp}

% \begin{obs}
%     Example \ref{ex:Counterexample} clarifies the usefulness of Assumption \ref{ass: K Global Lipschitz} (it gives at most linear growth of $\nabla K$ which implies that the gradient vector belongs to $\mathcal X$ for every $(x,u,p)\in \mathcal X$). It  shows that Assumptions \ref{ass:K_in_L1}, \ref{ass: K C1}, and \ref{ass:existence_optimizer} alone do not ensure that the formal grdient vector $\bigl\{(K_x(x,u),K_u(x,u),0)\bigr\}$ is $\mathcal X$-valued on the whole $D(J)$, where, $\mathcal X$ and $D(J)$ are specified by \eqref{eqn:X_Definition} and \eqref{eqn:J_Domain}, respectively. 
% \end{obs}

\smallskip

In the rest of the paper, unless explicitly
stated otherwise, we consider that Assumptions 
\ref{ass:K_in_L1}-\ref{ass: K Global Lipschitz} are satisfied. Then, $\partial\Phi(\xi)$ is nonempty, is a singleton
for every $\xi\in\mathcal X$, and we identify it with its
unique element in $\mathcal X$. Thus
%, consistently with \eqref{eq:DomainX}, 
$\partial\Phi:\mathcal X\to\mathcal X$
is the single-valued operator
\begin{equation}\label{eq:subgradient_singleton_cont}
\partial\Phi(\xi)=
\begin{pmatrix}
K_x(\widetilde x+\bar x,u)\\
K_u(\widetilde x+\bar x,u)\\
0
\end{pmatrix},
\qquad \xi=(\widetilde x,u,p)\in\mathcal X,
\end{equation}
where the componentwise identities hold for almost every
$t\in(0,T)$. By Assumption \ref{ass: K Global Lipschitz}, the
identification \eqref{eq:subgradient_singleton_cont},
and using the Lipschitz estimate of $K$ from Assumption \ref{ass: K Global Lipschitz}, we obtain the first inequality below; therefore, $\partial\Phi$ is Lipschitz.  
\begin{align}
\|\partial\Phi(\xi)-\partial\Phi(\zeta)\|_\mathcal X
&\leq L_K\|\xi-\zeta\|_\mathcal X,
\label{eq:subgradient_lipschitz_cont}\\
\langle\partial\Phi(\xi)-\partial\Phi(\zeta),\xi-\zeta\rangle_\mathcal X
&\geq0,\qquad \xi,\zeta\in \mathcal X.
\label{eq:subgradient_monotonicity_cont}
\end{align}
The second inequality follows from the monotonicity of $\partial\Phi$ established in Proposition \ref{prn:Phi_max_monotone}. We also have the following result. 
\begin{prn}\label{prn:minJplusMaxMon}
    %Suppose that Assumptions 
    %\ref{ass:K_in_L1}-\ref{ass: K Global Lipschitz} are satisfied. Then $-J+\partial\Phi$ is a maximal monotone operator and its domain is $D(J)$. 
    The operator $-J+\partial\Phi$ is maximal monotone on $\mathcal X$, and its domain is $D(J)$. 
\end{prn}
\begin{proof}
Assumption \ref{ass: K Global Lipschitz} holds, therefore we have \eqref{eq:DomainX}. 
Because $-J$ is maximal monotone (Proposition \ref{prn:-J_max_monotone}), $\Phi$ a lower semicontinuous proper
and convex function (Proposition \ref{prn:Phi_max_monotone}), and 
$D(-J)\cap \operatorname{int}D(\Phi)
= D(J)\neq\varnothing$, \cite[Theorem 2.62, p. 101]{BarbuPrecupanu2012} gives that 
$-J+\partial\Phi$ is maximal monotone (see also \cite[Theorem 1, p. 76]{Rockafellar1970Sums}). Then $D(-J+\partial\Phi)=D(-J)\cap D(\partial\Phi)=D(J)\cap \mathcal X=D(J)$. 
\end{proof}
\subsection{The weighted evolution equation}
\label{sect:weighted_evolution}
As in \eqref{eq:PDGCT_variational}, we distinguish physical time $t\in[0,T]$ from algorithmic time $\tau\geq 0$. Dots denote physical-time derivatives; primes denote algorithmic-time derivatives. 
Recall the homogenization $x=\widetilde x+\bar x$ 
introduced before in \eqref{eqn:J_Definition}, where the fixed lift $\bar x$ satisfies $\bar x\in H^1(0,T)$ and $\bar x(0)=x_0$. Thus
%\[
$\xi(\tau)=(\widetilde x(\cdot,\tau),u(\cdot,\tau),p(\cdot,\tau))$,
\qquad 
$x=\widetilde x+\bar x$,
%\] 
%
and applied to the third component of \eqref{eq:PDGCT_explicit} gives  
\begin{equation}\label{eq:homogenization_residual_cont}
-\dot x+Ax+Bu
=-\dot{\widetilde x}+A\widetilde x+Bu 
+A\bar x-\dot{\bar x}.
\end{equation}
Thus, the homogeneous-domain realization has constant forcing
\begin{equation}\label{eq:lift_forcing_cont}
h_{\bar x}:=(0,0,A\bar x-\dot{\bar x})\in \mathcal X.
\end{equation}
There is no algorithmic-time derivative of the lift since $\bar x$ is independent of $\tau$. The domain \eqref{eqn:J_Domain} already incorporates the
homogeneous boundary condition $\widetilde x(0)=0$, while the functional \eqref{eqn: Phi_Definition} evaluates the cost at the physical state $x=\widetilde x+\bar x$. One may choose
$\bar x(t)=e^{At}x_0$, in which case $h_{\bar x}=0$. The following results hold for every fixed lift
$\bar x\in H^1(0,T;\mathbb R^n)$ satisfying $\bar x(0)=x_0$, with the corresponding forcing $h_{\bar x}$ retained in the evolution equation. 

The positive definite weighting matrices $X,U,P$, introduced after \eqref{eq:PDGCT_variational}, are taken to be symmetric positive
definite, consistently with the Hamiltonian formulation in
\cite[(16), (21)]{vanderSchaft2026PDGJuly}. They induce
\begin{equation}
\label{eq:W_inner_product}
W:=\operatorname{diag}(X,U,P),\qquad
\langle\xi,\eta\rangle_W:=\langle W\xi,\eta\rangle_\mathcal X.
\end{equation}
% {\rev Ups: we have $X$ space and $X$ weighting matrix introduced
% after \eqref{eq:PDGCT_variational}! To change the name of the weighting matrix?}
%
Let $\mathcal X_W$ be $\mathcal X$ endowed with the inner product \eqref{eq:W_inner_product}. We have
\begin{equation}\label{eq:weighted_norm_bounds_cont}
w_-\|\xi\|_{\mathcal X}^2\leq\|\xi\|_W^2
\leq w_+\|\xi\|_{\mathcal X}^2,\qquad \xi\in\mathcal X,
\end{equation}
where  $w_-:=\lambda_{\min}(W)>0$, $w_+:=\lambda_{\max}(W)$. Using \eqref{eq:homogenization_residual_cont} and
\eqref{eq:subgradient_singleton_cont}, the formal equations
\eqref{eq:PDGCT_explicit} become
\begin{equation}\label{eq:weighted_optimizer_cont}
W\frac{d\xi}{d\tau}=J\xi-\partial\Phi(\xi)+h_{\bar x}.
\end{equation}
Define 
\begin{equation}\label{eq:weighted_operator_cont}
\begin{aligned}
\mathcal A_W\xi
:=W^{-1}\bigl(-J\xi+\partial\Phi(\xi)-h_{\bar x}\bigr),
\end{aligned}
\end{equation}
with $D(\mathcal A_W):=D(J)$. Then \eqref{eq:weighted_optimizer_cont} can be written as
%\[
$\displaystyle 
\frac{d\xi}{d\tau}+\mathcal A_W\xi=0$.
%\]
For $W=I$ and $\bar x(t)=e^{At}x_0$, this is the operator
$-J+\partial\Phi$ in Proposition \ref{prn:minJplusMaxMon}. Note that skew-adjointness of $W^{-1}J$ in $\mathcal X_W$, with domain $D(J)$, follows from Proposition \ref{prn:-J_max_monotone}, while Theorem \ref{thm: Theorem 2.43}, applied on $\mathcal X_W$, gives maximal monotonicity of the weighted subdifferential. Hence we have the following corollary of
Proposition \ref{prn:minJplusMaxMon}. 
%Therefore one can state the next result. %that includes the weights and the affine forcing for an arbitrary lift. 
\begin{cor}\label{prn:weighted_maximal_cont}
The operator $\mathcal A_W$ is maximal monotone
on $\mathcal X_W$, and its domain is dense in $\mathcal X_W$.
\end{cor}
% \begin{proof} \\
% %{\rev To be detailed later if needed.}
% \end{proof}
%
\subsection{Global solutions in algorithmic time}
\label{sect:global_solutions_cont}
We call $\xi$ a \textit{strong solution} if it is a locally absolutely
continuous $\mathcal X_W$-valued function such that, for a.e. $\tau \geq 0$, 
$\xi(\tau)\in D(J)$ and it satisfies \eqref{eq:weighted_optimizer_cont} in $\mathcal X_W$. A \textit{mild solution} with initial value $\xi_0\in\mathcal X$
is a limit of strong solutions $\xi_n$, with
$\xi_n(0)\in D(J)$ and $\xi_n(0)\to\xi_0$ in $\mathcal X_W$,
where $\xi_n\to\xi$ uniformly in $\mathcal X_W$ on every
bounded interval of algorithmic time.  
The following result establishes global well-posedness and specifies the regularity for initial data in $D(J)$. 

\begin{thm}\label{thm:global_wellposed_cont}
    For every $\xi_0\in\mathcal X$, equation
    \eqref{eq:weighted_optimizer_cont} has a unique global mild solution $\xi\in C([0,\infty);\mathcal X_W)$ with $\xi(0)=\xi_0$, and 
%\[
$S_W(\tau)\xi_0:=\xi(\tau;\xi_0)$,
%\qquad 
$\tau\geq 0$,
%\]
form a \textit{continuous semigroup of contractions} on $\mathcal X_W$ 
(see \cite[Definition 4.4, p. 154]{Barbu2010}).
Moreover
\begin{equation}\label{eq:contraction_cont}
	\|S_W(\tau)\xi_0-S_W(\tau)\eta_0\|_W
	\leq\|\xi_0-\eta_0\|_W
\end{equation}
for all $\xi_0,\eta_0\in\mathcal X$
and $\tau\geq 0$. The resolvent formula
\begin{equation}\label{eq:resolvent_limit_cont}
	S_W(\tau)\xi_0
	=
	\lim_{N\to\infty}
	\left(I+\frac{\tau}{N}\mathcal A_W\right)^{-N}\xi_0
\end{equation}
holds in $\mathcal X_W$, uniformly for $\tau\in[0,b]$
for every $b>0$ and every fixed $\xi_0\in\mathcal X$. If $\xi_0\in D(J)$, then the solution is strong and $\xi\in W^{1,\infty}(0,b;\mathcal X_W)$ for $b>0$, $\xi(\tau)\in D(J)$, %for every
$\tau\geq 0$, and it satisfies
\begin{equation}\label{eq:bound}
    \|\mathcal A_W\xi(\tau)\|_W
	\leq\|\mathcal A_W\xi_0\|_W
	\text{ for every }\tau\geq 0.
\end{equation}
The physical-time traces
$\widetilde x(0,\tau)=0$ and $p(T,\tau)=0$
are defined and hold for every $\tau\geq0$.
\end{thm}
\begin{proof}
By Proposition \ref{prn:weighted_maximal_cont},
$\mathcal A_W$ is maximal monotone on $\mathcal X_W$, with $D(\mathcal A_W)=D(J)$ and $\overline{D(\mathcal A_W)}=\mathcal X_W$. We apply the results to $	\xi'+\mathcal A_W\xi=0.$ 
The term $h_{\bar x}$ is already included in
$\mathcal A_W$. 
Existence and uniqueness of mild solutions, together with \eqref{eq:resolvent_limit_cont}, follow from \cite[Corollary 4.1 and Theorem  4.3,
pp. 131-132]{Barbu2010}. These results apply on every bounded interval of algorithmic time. Uniqueness makes the solutions consistent on overlapping intervals, yielding a global solution. For $\xi_0\in D(\mathcal A_W)$, \cite[Theorem 4.5, p. 141]{Barbu2010},
applied with forcing $f=0$ and $\omega =0$, gives
$\xi\in W^{1,\infty}(0,b;\mathcal X_W)$
for every $b>0$. 
For two strong solutions $\xi$ and $\eta$, the
differentiation rule 
\cite[Lemma 4.1, p. 128]{Barbu2010}, 
skew-symmetry of $J$, and
\eqref{eq:subgradient_monotonicity_cont} yield
\[
\begin{aligned}
	\frac{1}{2}\frac{d}{d\tau}\|\xi-\eta\|_W^2
	=-\bigl\langle
	\partial\Phi(\xi)-\partial\Phi(\eta),
	\xi-\eta
	\bigr\rangle_{\mathcal X}
	\leq 0, \mbox{ a.e. } \tau.
\end{aligned}
\]
For $\xi_0,\eta_0\in D(J)$, subtracting the two evolution equations cancels the common affine forcing. Taking the corresponding mild-solution limits gives \eqref{eq:contraction_cont} for all initial data. 
Then using domain invariance and \eqref{eqn:J_Domain}, give \eqref{eq:contraction_cont} and \eqref{eq:bound}. Note that, for initial data $\xi_0\in D(J)$, domain invariance follows from
\cite[proof of Theorem 4.6, p. 145]{Barbu2010}. For each fixed $\tau$, the traces are evaluated using the absolutely continuous representatives of
$\widetilde x(\cdot,\tau),p(\cdot,\tau)
\in H^1(0,T;\mathbb R^n)$; see \cite[Theorem 1.17, pp. 23-24]{Barbu2010}. Finally, $S_W(0)=I$, and $\tau\mapsto S_W(\tau)\xi_0$ is continuous for every $\xi_0\in\mathcal X$.
Together with \eqref{eq:contraction_cont}, these properties establish that $S_W$ is a semigroup of contractions on $\mathcal X_W$.
\end{proof}

\section{Convergence}
We first identify 
%the stationary point of the semigroup.
the equilibria of the weighted optimizer evolution in algorithmic time 
in \eqref{eq:weighted_optimizer_cont} with the optimal solutions of \eqref{eq:ocp_cost}. 
%Then we distinguish Lyapunov stability from convergence. 
%
%\subsection{The optimal solution}
The next result follows from the minimum principle in \cite[Section 4.1.3, Theorem 4.5, pp. 243-244]
{BarbuPrecupanu2012} and the extremality conditions in \cite[Section 4.1.3]
{BarbuPrecupanu2012}, with the costate sign adjusted to match. 
\begin{prn}\label{prn:equilibrium_cont}
	There is exactly one equilibrium
	$\xi_\star=(x_\star-\bar x,u_\star,p_\star)\in D(J)$ of
	\eqref{eq:weighted_optimizer_cont}. 
    The pair $(x_\star,u_\star)$ is the unique optimal solution of \eqref{eq:ocp_cost}, and its costate is given by
    \begin{equation}\label{eq:costate_formula_cont}
	p_\star(t)=\int_t^T e^{A^\top(s-t)}
	K_x\bigl(x_\star(s),u_\star(s)\bigr)\,ds,
	\qquad t\in[0,T]. 
\end{equation}
In particular, $\{\xi\in D(\mathcal A_W):\mathcal A_W\xi=0\} =\{\xi_\star\}$.
\end{prn}
\begin{proof} A proof is provided in Appendix \ref{App: proof_Proposition4}.
\end{proof}

\subsection{Dissipation of the shifted Hamiltonian}
The shifted Hamiltonian functional and the associated
dissipation inequality are given in
\cite[(22)-(25)]
{vanderSchaft2026PDGJuly}.
We now establish the corresponding integral identity
for every mild solution of
\eqref{eq:weighted_optimizer_cont}. 
%in the Hilbert space realization developed above. 
Note that the energy variables $(z,v,r)$ introduced in \eqref{eq:Hamiltonian_functional} are related to the homogenized variable $\xi$ by
$(z,v,r)=W\xi+(X\bar x,0,0)$. 
After subtracting the equilibrium, this relation becomes
\begin{equation}
\label{eq:RelationHomVariables}
 (z-z_\star,v-v_\star,r-r_\star)=W(\xi-\xi_\star). 
\end{equation}
We therefore work directly with $\xi\in\mathcal X_W$ and $\xi_\star=(x_\star-\bar x,u_\star,p_\star)\in D(J)$, where $\xi_\star$ is the unique equilibrium given by
Proposition \ref{prn:equilibrium_cont}.  
Then the shifted Hamiltonian \eqref{eq:shifted_Hamiltonian}
in the homogenized coordinates
is %represented by
\begin{equation}
    \label{eq:energy_diss_cont}
    V_W(\xi):=\mathcal H_e\bigl(X(\widetilde x+\bar x),Uu,Pp\bigr)=\tfrac12\|\xi-\xi_\star\|_W^2.
\end{equation}
Define 
\begin{equation}
    \label{eq:DcalDiff}
    \mathscr D(\xi):=
\bigl\langle\partial\Phi(\xi)-\partial\Phi(\xi_\star),
\xi-\xi_\star\bigr\rangle_{\mathcal X}.
\end{equation}
Both functionals in \eqref{eq:energy_diss_cont} and \eqref{eq:DcalDiff} are defined
on the whole space $\mathcal X$, since $\partial\Phi$ is single-valued
and Lipschitz (see \eqref{eq:subgradient_singleton_cont}
and \eqref{eq:subgradient_lipschitz_cont}).
Using $x=\widetilde x+\bar x$ and the representation
\eqref{eq:subgradient_singleton_cont} we have 
\begin{equation}\label{eq:dissipation_physical_cont}
	\begin{aligned}
		\mathscr D(\xi)
		=\int_0^T\Big[
		&\bigl(K_x(x,u)-K_x(x_\star,u_\star)\bigr)^\top
		(x-x_\star)\\
		&+\bigl(K_u(x,u)-K_u(x_\star,u_\star)\bigr)^\top
		(u-u_\star)\Big]\,dt.
	\end{aligned}
\end{equation}
%with the functions inside the integral evaluated at physical time $t$. 
By Assumption \ref{ass: K C1}, $K$ is differentiable and strictly convex. Its gradient is therefore strictly
monotone
% : for $a,b\in\mathbb R^{n+m}$, $a\neq b$,
% %\[
% $\bigl(\nabla K(a)-\nabla K(b)\bigr)^\top(a-b)>0$ 
% %\]
(see e.g. \cite[Theorem 2.4.4(ii),
pp. 85]{Zalinescu2002}). Thus, the integrand in \eqref{eq:dissipation_physical_cont}
is nonnegative and vanishes when 
$(x(t),u(t))=(x_\star(t),u_\star(t))$, so  
\begin{equation}\label{eq:diss_zero_cont}
	%\begin{aligned}
	\mathscr D(\xi)\geq 0, \;
		\mathscr D(\xi)=0
		\Longleftrightarrow
		(x,u)=(x_\star,u_\star)
		\;\text{a.e. on }(0,T).
	%\end{aligned}
\end{equation}
Note that $\mathscr D(\xi)=0$ does not impose any condition on the costate $p$.
\begin{prn}\label{prn:energy_cont}
	The mild solution
	$\xi(\tau)=S_W(\tau)\xi_0$ of
	\eqref{eq:weighted_optimizer_cont} satisfies
	\begin{equation}\label{eq:energy_identity_cont}
		V_W(\xi(\tau))
		+\int_0^\tau\mathscr D(\xi(s))\,ds
		=V_W(\xi_0),
		\qquad \tau\geq 0
	\end{equation}
    for every $\xi_0\in\mathcal X$.  
    Moreover,  $\xi_\star$ is Lyapunov stable in
$\mathcal X_W$, every orbit is bounded in $\mathcal X_W$,
and $\mathscr D(\xi(\cdot))\in L^1(0,\infty;\mathbb R)$. 
    Furthermore, $V_W(\xi(\cdot))\in C^1([0,\infty);\mathbb R)$ and satisfies
    \begin{equation}
        \label{eq:VWde_Xideeta}
        \frac{d}{d\tau}V_W(\xi(\tau)) =-\mathscr D(\xi(\tau)), \qquad \tau > 0. 
    \end{equation}
\end{prn}
\begin{proof}
First take $\xi_0\in D(J)$. From Theorem \ref{thm:global_wellposed_cont}, $\xi$ is a strong solution and $\xi(\tau)\in D(J)$ for every $\tau\geq 0$, $\xi\in W^{1,\infty}(0,b;\mathcal X_W)$ for every $b>0$. Since $D(J)$ is linear and $\xi_\star\in D(J)$, then $\eta (\tau)=\xi (\tau)-\xi_\star \in D(J)$. Using 
$W\eta'(\tau)=J\eta(\tau)- \bigl(\partial\Phi(\xi(\tau))-\partial\Phi(\xi_\star)\bigr)$ and and the skew-symmetry of $J$ (see Proposition \ref{prn:-J_max_monotone}), 
we obtain
\begin{align*}
\frac{d}{d\tau}V_W(\xi(\tau))
&=\langle W\eta'(\tau),\eta(\tau)\rangle_{\mathcal X}\\
&=\langle J\eta(\tau),\eta(\tau)\rangle_{\mathcal X}
-\mathscr D(\xi(\tau))
=-\mathscr D(\xi(\tau))
\end{align*}
for a.e. $\tau>0$. Note that $\xi$, and hence $\eta$,
is Lipschitz on $[0,b]$. Thus $V_W(\xi(\cdot))$ is Lipschitz, and therefore absolutely continuous, on every bounded interval of algorithmic time, so integration gives
\eqref{eq:energy_identity_cont}. \\
Now let $\xi_0\in\mathcal X$. Choose $\xi_{0,n}\in D(J)$ such that $\xi_{0,n}\to\xi_0$ in $\mathcal X_W$, and take  
$\xi_n(\tau):=S_W(\tau)\xi_{0,n}$. Using \eqref{eq:contraction_cont}, for every $b>0$,
\[
\sup_{0\leq s\leq b}\|\xi_n(s)-\xi(s)\|_W
\leq\|\xi_{0,n}-\xi_0\|_W
\longrightarrow0.
\]
The same uniform convergence holds in $\mathcal X$
by \eqref{eq:weighted_norm_bounds_cont}. Now, \eqref{eq:subgradient_lipschitz_cont} gives 
\begin{equation}\label{eq:dissipation_continuity_cont}
|\mathscr D(\xi)-\mathscr D(\zeta)|\leq
L_K\bigl(\|\xi-\xi_\star\|_{\mathcal X}
+\|\zeta-\xi_\star\|_{\mathcal X}\bigr)
\|\xi-\zeta\|_{\mathcal X},
\end{equation}
for all $\xi,\zeta\in\mathcal X$, so $\mathscr D$ is Lipschitz on every bounded subset of $\mathcal X$. For fixed $b$, $\xi_n$ and $\xi$ are uniformly bounded on $[0,b]$, so
$\mathscr D(\xi_n(\cdot))\to\mathscr D(\xi(\cdot))$
uniformly on $[0,b]$. Taking the limit in
\eqref{eq:energy_identity_cont} for $\xi_n$ proves the same identity for $\xi$ on $[0,b]$. Since $b>0$ is arbitrary, it holds for every $\tau\geq 0$. Using \eqref{eq:diss_zero_cont} and
\eqref{eq:energy_identity_cont}, we obtain 
%\[
$\|\xi(\tau)-\xi_\star\|_W
\leq\|\xi_0-\xi_\star\|_W$,
$\tau\geq 0$.
%\] 
Therefore, we have Lyapunov stability of $\xi_\star$ and boundedness of every orbit $\{\xi(\tau)=S_W(\tau)\xi_0:\tau\geq0\}$, $\xi_0\in\mathcal X$. 
Then \eqref{eq:energy_identity_cont} gives, for every $\tau\geq0$,
$\displaystyle \int_0^\tau\mathscr D(\xi(s))\,ds
=V_W(\xi_0)-V_W(\xi(\tau))\leq V_W(\xi_0)$,
because $V_W(\xi(\tau))\geq0$. 
Since $\mathscr D(\xi)\geq 0$, 
\[
\int_0^\infty\mathscr D(\xi(s))\,ds
=\lim_{\tau\to\infty}\int_0^\tau\mathscr D(\xi(s))\,ds
\leq V_W(\xi_0)<\infty  
\]
by monotone convergence, so $\mathscr D(\xi(\cdot))\in L^1(0,\infty;\mathbb R)$. 
Finally, $\xi\in C([0,\infty);\mathcal X_W)$ and
$\mathscr D$ is continuous by 
\eqref{eq:dissipation_continuity_cont}, so 
$\mathscr D(\xi(\cdot))\in C([0,\infty);\mathbb R)$. Then, by \eqref{eq:energy_identity_cont} and the fundamental theorem of calculus,
$V_W(\xi(\cdot))\in C^1([0,\infty);\mathbb R)$,
and \eqref{eq:VWde_Xideeta} holds.
\end{proof}

\subsection{Convergence of relatively compact orbits}
\label{sect:compact_convergence_cont}

We apply the invariance principle stated in
\cite[Theorem 2.1]{SlemrodInfante1969}. 
% This principle applies to continuous dynamical systems on a Banach space and requires relative compactness of the positive orbit. 
We verify its hypotheses for the semigroup of Theorem \ref{thm:global_wellposed_cont} and identify the invariant subset of the zero-dissipation set for mild solutions. 
This gives a convergence result for the Hilbert
space realization considered here, under a compactness condition (see the discussion following \cite[Proposition 3.2]{vanderSchaft2026PDGJuly}). In this subsection, compactness and (strong) convergence
refer to the norm topology of $\mathcal X$. From \eqref{eq:weighted_norm_bounds_cont}, these notions are unchanged when $\mathcal X$ is equipped with norm $\|\cdot\|_W$. 

We first establish an identity (needed in the proof of Theorem \ref{thm:compact_convergence_cont}) that allows to 
use the operator $J$ along mild solutions. 

\begin{lemma}\label{lem:integrated_mild_cont}
For every $\xi_0\in\mathcal X$, the mild solution
$\xi(\tau)=S_W(\tau)\xi_0$, $t\geq 0$, satisfies
$\displaystyle \int_0^\tau \xi(s)\,ds\in D(J)$, and 
\begin{equation}
\label{eq:integrated_mild_cont}
W\bigl(\xi(\tau)-\xi_0\bigr)
=J\left(\int_0^\tau \xi(s)\,ds\right)
-\int_0^\tau\partial\Phi(\xi(s))\,ds
+\tau h_{\bar x}.
\end{equation}
\end{lemma}
\begin{proof}
%A proof is provided in Appendix \ref{prooflem:integrated_mild_cont}. 
First suppose that $\xi_0\in D(J)$. By Theorem
\ref{thm:global_wellposed_cont}, $\xi$ is a strong
solution. For any $v\in D(J)$, pairing
\eqref{eq:weighted_optimizer_cont} with $v$, using
the skew-adjointness of $J$ 
%from Proposition \ref{prn:-J_max_monotone}, 
and integrating in algorithmic time give
\begin{equation}\label{eq:weak_integrated_mild_cont}
	\begin{aligned}
		&\left\langle
		W\bigl(\xi(\tau)-\xi_0\bigr)
		+\int_0^\tau\partial\Phi(\xi(s))\,ds
		-\tau h_{\bar x},v
		\right\rangle_{\mathcal X}\\
		&\hspace{15mm}
		=-\left\langle\int_0^\tau\xi(s)\,ds,Jv
		\right\rangle_{\mathcal X}.
	\end{aligned}
\end{equation}
Indeed, the %scalar function
$\tau\mapsto\langle W\xi(\tau),v\rangle_{\mathcal X}$
is absolutely continuous on every bounded interval,
so its almost-everywhere derivative can be integrated. 
Now let $\xi_0\in\mathcal X$. Since $D(J)$ is dense,
choose $\xi_{0,k}\in D(J)$ such that
$\xi_{0,k}\to\xi_0$ in $\mathcal X_W$, and put
$\xi_k(\tau)=S_W(\tau)\xi_{0,k}$. By
\eqref{eq:contraction_cont}, for every $b>0$,
%\[
$\displaystyle 
\sup_{0\leq s\leq b}
\|\xi_k(s)-\xi(s)\|_W
\leq\|\xi_{0,k}-\xi_0\|_W
\longrightarrow 0$. 
%\]
Then 
\eqref{eq:weighted_norm_bounds_cont} gives the 
uniform convergence in $\mathcal X$. Moreover,
\eqref{eq:subgradient_lipschitz_cont} implies that
$\partial\Phi(\xi_k(\cdot))$ converges uniformly to
$\partial\Phi(\xi(\cdot))$ on $[0,b]$ in $\mathcal X$.
We can therefore pass to the limit in
\eqref{eq:weak_integrated_mild_cont}. Thus \eqref{eq:integrated_mild_cont} holds for every mild solution, every
$v\in D(J)$, and every $\tau\geq0$. 
For each fixed $\tau$, the vector paired with $v$
on the left-hand side of
\eqref{eq:weak_integrated_mild_cont} is in
$\mathcal X$. By the definition of the adjoint,
\eqref{eq:weak_integrated_mild_cont} implies
$\displaystyle
\int_0^\tau\xi(s)\,ds\in D(J^*)$.
Since $D(J^*)=D(J)$ and $J^*=-J$, 
Proposition \ref{prn:-J_max_monotone} gives
\eqref{eq:integrated_mild_cont}.
\end{proof}
The next result states that the state, control, and costate converge strongly in their respective $L^2(0,T)$ spaces to $x_\star$, $u_\star$, and $p_\star$ if the corresponding positive orbit is relatively compact. 
\begin{thm}\label{thm:compact_convergence_cont}
	Let $\xi_0\in\mathcal X$. If the positive orbit
	%\[
	$\mathcal O^+(\xi_0):=
	\{S_W(\tau)\xi_0:\tau\geq 0\}$
	%\]
	is relatively compact in $\mathcal X$, then
	\begin{equation}\label{eq:compact_convergence_cont}
		\lim_{\tau\to\infty}
		\|S_W(\tau)\xi_0-\xi_\star\|_{\mathcal X}=0,
	\end{equation}
\end{thm}
\begin{proof}
    We verify the hypotheses of
\cite[Theorem 2.1]{SlemrodInfante1969}
on the space $\mathcal X_W$. 
First, Theorem \ref{thm:global_wellposed_cont} and
\eqref{eq:contraction_cont}, gives that 
$(\tau,\zeta)\mapsto S_W(\tau)\zeta$ is a continuous dynamical system on $\mathcal X_W$ in the sense of \cite[Definition 2.1]{SlemrodInfante1969}. Then, $V_W$ from \eqref{eq:energy_diss_cont} is
continuous and nonnegative on $\mathcal X_W$.
For every $\zeta\in\mathcal X$, the identity
\eqref{eq:energy_identity_cont} gives
\[
\frac{V_W(S_W(h)\zeta)-V_W(\zeta)}{h}
=
-\frac{1}{h}\int_0^h
\mathscr D(S_W(s)\zeta)\,ds,
\qquad h>0.
\] 
Since $S_W(\cdot)\zeta$ continuous, using 
\eqref{eq:dissipation_continuity_cont} 
we have 
\begin{equation}\label{eq:lyapunov_semigroup_derivative_cont}
	\lim_{h\downarrow0}
	\frac{V_W(S_W(h)\zeta)-V_W(\zeta)}{h}
	=-\mathscr D(\zeta)\leq0.
\end{equation} 
Therefore, $V_W$ satisfies the
Lyapunov-functional conditions of
\cite[Definitions 2.4-2.5] {SlemrodInfante1969}
on $\mathcal X_W$. \\
Now, define the zero-dissipation set 
$\mathcal Z:=\{\zeta\in\mathcal X:
\mathscr D(\zeta)=0\}$. We show that its largest positively invariant subset is $\{\xi_\star\}$
(see \cite[Definition 2.2]{SlemrodInfante1969}). 
Let $\zeta\in\mathcal X$ satisfy $\mathscr D(S_W(s)\zeta)=0$, $s\geq 0$. 
Using \eqref{eq:diss_zero_cont}, we have that for every $s\geq0$, the homogenized state and control components of
$S_W(s)\zeta$ satisfy
$\widetilde x(\cdot,s)=x_\star-\bar x$ and
$u(\cdot,s)=u_\star$, respectively. Then $\eta(s):=S_W(s)\zeta-\xi_\star=(0,0,\rho(s))$, 
$\rho(s)\in L^2(0,T;\mathbb R^n)$. 
Moreover, \eqref{eq:subgradient_singleton_cont}
implies $\partial\Phi(S_W(s)\zeta)
=\partial\Phi(\xi_\star)$, $s\geq 0$. Subtracting \eqref{eq:integrated_mild_cont} applied to  $\xi_\star$
from that applied to $S_W(\cdot)\zeta$, and using the linearity of $J$, gives
$\displaystyle \int_0^s\eta(a)\,da\in D(J)$ and 
\begin{equation}
\label{eq:zero_dissipation_integrated_cont}
    W\bigl(\eta(s)-\eta(0)\bigr)
=J\left(\int_0^s\eta(a)\,da\right),\qquad s\geq 0. 
\end{equation}
The first two components of $\eta$ vanish.
Therefore, by \eqref{eqn:J_Definition}, 
the third component of
\eqref{eq:zero_dissipation_integrated_cont} is 
$P\bigl(\rho(s)-\rho(0)\bigr)=0$. 	Since $P$ is positive definite, $\rho(s)=\rho(0)$
in $L^2(0,T;\mathbb R^n)$ for every $s\geq 0$.
Thus $\eta(s)=\eta(0)$ for every $s\geq 0$. 

For any $s>0$, $\displaystyle \int_0^s\eta(a)\,da\in D(J)$ (see above \eqref{eq:zero_dissipation_integrated_cont}) 
gives $s\eta(0)\in D(J)$, so $\eta(0)\in D(J)$, and \eqref{eq:zero_dissipation_integrated_cont} gives $J\eta(0)=0$.  
In particular, if
$\rho_0:=\rho(0)$, then
\eqref{eqn:J_Definition} and \eqref{eqn:J_Domain}
imply
\[
\rho_0\in H^1(0,T;\mathbb R^n),\qquad
\dot\rho_0+A^\top\rho_0=0,\qquad
\rho_0(T)=0.
\]
The terminal-value problem has the unique solution
%\[
$\rho_0(t)=e^{A^\top(T-t)}\rho_0(T)=0$,
$t\in[0,T]$.
%\]
It follows that $\zeta=\xi_\star$.

Conversely, $\xi_\star$ is an equilibrium by
Proposition \ref{prn:equilibrium_cont}, and
$\mathscr D(\xi_\star)=0$. Thus
$\{\xi_\star\}$ is positively invariant.
It is also invariant in the sense of \cite[Definition 2.3]{SlemrodInfante1969}, since the constant equilibrium trajectory is defined
for all algorithmic times in $\mathbb R$. 
Every invariant subset of $\mathcal Z$ is
positively invariant, so the largest invariant subset of $\mathcal Z$ is $\{\xi_\star\}$. 
By the assumed relative compactness of $\mathcal O^+(\xi_0)$ in $\mathcal X$ and \eqref{eq:weighted_norm_bounds_cont}, its closure is compact in $\mathcal X_W$. 
All hypotheses of \cite[Theorem 2.1]{SlemrodInfante1969} are now satisfied, so 
%\[
$\|S_W(\tau)\xi_0-\xi_\star\|_W
\longrightarrow 0$
as $\tau\to\infty$. 
%\]
Then \eqref{eq:weighted_norm_bounds_cont} proves
\eqref{eq:compact_convergence_cont}. 
%Since $x=\widetilde x+\bar x$ and the lift is fixed, convergence of $\widetilde x$ to $x_\star-\bar x$ is equivalent to convergence of $x$ to $x_\star$. 
Since $\widetilde x-(x_\star-\bar x)=x-x_\star$,
% the definition \eqref{eqn:X_Definition} of $\mathcal X$ gives 
we obtain 
$L^2$ convergence of $x$, $u$, and $p$ to $x_\star$, $u_\star$, and $p_\star$, respectively.  
\end{proof}

% {\rev 
% \begin{obs}
% Note that in the proof of the above theorem we use the terminal condition and
% controllability of $(A,B)$ is not requires.
% \end{obs}
% } 

We next give a sufficient condition for the
relative compactness required in Theorem
\ref{thm:compact_convergence_cont}. 
% The arguments follow \cite[Section 5, p. 7]{SlemrodInfante1969}:
% an orbit is bounded in a stronger space which
% embeds compactly into the space on which the
% invariance principle is applied. We shall use the compact embedding $H^1(0,T)\hookrightarrow L^2(0,T)$. 

\begin{cor}\label{cor:control_regularity_compact_cont}
	Let $\xi_0\in D(J)$ and let
	$\xi(\tau)=S_W(\tau)\xi_0$. Suppose, in addition,
	that the control component satisfies
	\begin{equation}\label{eq:uniform_control_H1_cont}
		\begin{aligned}
			&u(\cdot,\tau)\in H^1(0,T;\mathbb R^m),
			\qquad \tau\geq0,\\
			&\sup_{\tau\geq0}
			\|u(\cdot,\tau)\|_{H^1(0,T;\mathbb R^m)}
			<\infty.
		\end{aligned}
	\end{equation}
	Then $\mathcal O^+(\xi_0)$ is relatively compact
	in $\mathcal X$, and
	\eqref{eq:compact_convergence_cont} holds.
\end{cor}

\begin{proof}
    Proposition \ref{prn:energy_cont} and
\eqref{eq:weighted_norm_bounds_cont} give 
$\displaystyle \sup_{\tau\geq0}\|\xi(\tau)\|_{\mathcal X}<\infty$. Since $\xi_0\in D(J)$, Theorem
\ref{thm:global_wellposed_cont} gives
$\xi(\tau)\in D(J)$ for every $\tau\geq 0$. 
From \eqref{eq:weighted_operator_cont}, 
$J\xi(\tau)=
\partial\Phi(\xi(\tau))-h_{\bar x}
-W\mathcal A_W\xi(\tau)$. The Lipschitz estimate
\eqref{eq:subgradient_lipschitz_cont} implies 
$\|\partial\Phi(\xi(\tau))\|_{\mathcal X}
\leq \|\partial\Phi(0)\|_{\mathcal X}
+L_K\|\xi(\tau)\|_{\mathcal X}$. 
This inequality, together with \eqref{eq:bound}, $W$ bounded, and \eqref{eq:weighted_norm_bounds_cont}, gives 
\begin{equation}\label{eq:uniform_J_bound_cont}
	\sup_{\tau\geq0}
	\|J\xi(\tau)\|_{\mathcal X}<\infty.
\end{equation}
Using %the first and third components of 
\eqref{eqn:J_Definition}, boundedness of $\xi$ in $\mathcal X$ and \eqref{eq:uniform_J_bound_cont} 
we obtain 
\[
\sup_{\tau\geq0}
\left(
\|\widetilde x(\cdot,\tau)\|_{H^1(0,T;\mathbb R^n)}
+\|p(\cdot,\tau)\|_{H^1(0,T;\mathbb R^n)}
\right)<\infty.
\]
Together with \eqref{eq:uniform_control_H1_cont},
this shows that the orbit is bounded in 
$H^1(0,T;\mathbb R^n)
\times H^1(0,T;\mathbb R^m)
\times H^1(0,T;\mathbb R^n)$. 
%Since $T<\infty$, this product space embeds compactly into $\mathcal X$; see \cite[Theorem 8.8, pp. 212-213]{Brezis2011},applied componentwise. More precisely, 
\cite[Theorem 8.8, pp. 212-213]{Brezis2011}
gives the compact embedding $H^1(0,T)\hookrightarrow C([0,T])$,
and uniform convergence implies $L^2$ convergence because
$\|f\|_{L^2(0,T)}\leq\sqrt{T}\,\|f\|_{C([0,T])}$
for every $f\in C([0,T])$. Applying this result componentwise to the $2n+m$ scalar
coordinates of $\widetilde x$, $u$, and $p$, and successively extracting subsequences, yields a common subsequence along which all coordinates converge strongly in $L^2(0,T)$.
Hence $\mathcal O^+(\xi_0)$ is relatively compact in
$\mathcal X$, and Theorem \ref{thm:compact_convergence_cont} gives 
\eqref{eq:compact_convergence_cont}.
\end{proof}
The following result extends the convergence to every initial value in $\mathcal X$ (by standard density arguments and the contraction  \eqref{eq:contraction_cont}).

\begin{cor}\label{cor:dense_convergence_cont}
	Suppose that there exists a dense subset
	$\mathcal C\subset\mathcal X$ such that
	$\mathcal O^+(\zeta)$ is relatively compact in
	$\mathcal X$ for every $\zeta\in\mathcal C$.
	Then
	\eqref{eq:compact_convergence_cont} holds for
	every $\xi_0\in\mathcal X$.
	In particular, this conclusion holds if
	$\mathcal C\subset D(J)$ is dense in $\mathcal X$
	and the trajectory starting at every
	$\zeta\in\mathcal C$ satisfies
	\eqref{eq:uniform_control_H1_cont}. The bound in
	that condition may depend on $\zeta$.
\end{cor}
Next, we illustrate the theory with a strictly but not strongly convex $K$.
\begin{exmp}
\label{ex:logcosh_convergence}
Let $n=m=1$, $T>0$, and let $A,B,x_0\in\mathbb R$ be arbitrary. Choose $X=U=P=1$, so that $W=I$, and take $\bar x(t)=e^{At}x_0$. Then $h_{\bar x}=0$ in \eqref{eq:lift_forcing_cont}. Consider
	%\begin{equation}\label{eq:logcosh_cost}
		$K(x,u)=\log\cosh x+\frac12u^2$.
	%\end{equation}  
First we note that, $K$ is not strongly convex on $\mathbb R^2$. This would need
$\nabla^2K(x,u)\succeq\mu I$ for some $\mu>0$ and all $(x,u)$, whereas $K_{xx}(x,u)=1/\cosh^2x\to0$ as
$|x|\to\infty$. 
Since $0\le K(x,u)\le\frac12(x^2+u^2)$,
Assumption \ref{ass:K_in_L1} holds. 
$\nabla^2K(x,u)=\operatorname{diag}(1/\cosh^2x,1)$
is positive definite and has operator norm one. 
Thus Assumptions \ref{ass: K C1} and
\ref{ass: K Global Lipschitz} hold with $L_K=1$.
%
% The Hessian of $K$ is positive definite at every $(x,u)$, so
% $K$ is strictly convex and Assumption \ref{ass: K C1} holds. Moreover, its operator norm is at most one; hence $\nabla K$ is globally Lipschitz with $L_K=1$, as required by Assumption \ref{ass: K Global Lipschitz}. 
To verify Assumption \ref{ass:existence_optimizer},
associate to each $u\in\mathcal U=L^2(0,T)$ the unique feasible state
%\[
$\displaystyle 
x_u(t)=e^{At}x_0+\int_0^t e^{A(t-s)}Bu(s)\,ds$.
%\]
The map $u\mapsto x_u$ is continuous and affine from $L^2(0,T)$ to $H^1(0,T)$,  
% Since
% $|\log\cosh a-\log\cosh b|\le|a-b|$ and 
so the map
% $u\mapsto\frac12\|u\|_{L^2(0,T)}^2$ is continuous,
% the function
\begin{equation}
	\label{eq:cost_example_strictconvex}
u\longmapsto
\int_0^T
\left(\log\cosh\bigl(x_u(t)\bigr)+\frac12u(t)^2\right)\,dt
\end{equation}
is continuous on $L^2(0,T)$, finite-valued and convex. 
% Since $u\mapsto x_u$ is affine and $K$ is convex, this function is convex. Being convex and continuous in the $L^2(0,T)$ norm, it is weakly lower semicontinuous on $L^2(0,T)$. 
% Furthermore,
% \[
% \int_0^T
% \left(\log\cosh\bigl(x_u(t)\bigr)+\frac12u(t)^2\right)\,dt
% \ge \frac12\|u\|_{L^2(0,T)}^2.
% \]
It is bounded below by $\frac12\|u\|_{L^2}^2$,
hence coercive, and attains its minimum by
\cite[Corollary 3.23, p. 71]{Brezis2011}.
This proves Assumption \ref{ass:existence_optimizer}.
Proposition \ref{prn:equilibrium_cont} 
gives the unique equilibrium
$\xi_\star=(x_\star-\bar x,u_\star,p_\star)$. 
% The cost is finite at $u=0$, so its infimum is
% finite and every minimizing sequence is bounded.  Since $L^2(0,T)$ is reflexive and the functional \eqref{eq:cost_example_strictconvex} 
% is finite-valued, convex, continuous, and coercive, it attains its minimum by
% \cite[Corollary 3.23, p. 71]{Brezis2011}.
% The minimizing control and its associated feasible state therefore solve \eqref{eq:ocp_cost}, proving Assumption  \ref{ass:existence_optimizer}. Since all assumptions are satisfied, Proposition \ref{prn:equilibrium_cont} identifies the unique equilibrium
% $\xi_\star=(x_\star-\bar x,u_\star,p_\star)$.
%
Now, define
\begin{equation}\label{eq:logcosh_dense_set}
	\mathcal C:=
	\{(\widetilde x,u,p)\in D(J):u\in H^1(0,T)\}.
\end{equation}
By \eqref{eqn:J_Domain}, this set contains
$[C_c^\infty(0,T)]^3$ and is dense in
$\mathcal X=[L^2(0,T)]^3$. 
Fix $\xi_0\in\mathcal C$ and let $\xi$ be the strong
solution given by Theorem~\ref{thm:global_wellposed_cont}.
As in the derivation of \eqref{eq:uniform_J_bound_cont},
Proposition \ref{prn:energy_cont} and \eqref{eq:bound}
give uniform bounds for $\xi(\tau)$ and $J\xi(\tau)$.
The first component of \eqref{eqn:J_Definition}
then gives
$M_p:=\sup_{\tau\ge 0}
\|p(\cdot,\tau)\|_{H^1(0,T)}<\infty$. 
As $K_u(x,u)=u$, the control equation in 
\eqref{eq:weighted_optimizer_cont} is
$u'=-u-Bp$ in $L^2(0,T)$ for a.e. $\tau\ge 0$.
Variation of constants gives
\begin{equation}\label{eq:logcosh_control_formula}
	\begin{aligned}
		u(\cdot,\tau)
		=e^{-\tau}u(\cdot,0)
		-B\int_0^\tau e^{-(\tau-s)}p(\cdot,s)\,ds 
	\end{aligned}
\end{equation}
and we have 
%\[
$\displaystyle 
\left\|\int_0^\tau e^{-(\tau-s)}p(\cdot,s)\,ds\right\|_{H^1}
\le M_p(1-e^{-\tau})$.
%\]
Since $u(\cdot,0)\in H^1(0,T)$, we get 
$u(\cdot,\tau)\in H^1(0,T)$ for $\tau\ge 0$, and 
\begin{equation}\label{eq:logcosh_control_H1}
	\begin{aligned}
		\|u(\cdot,\tau)\|_{H^1}
		&\le e^{-\tau}\|u(\cdot,0)\|_{H^1}
		+|B|M_p(1-e^{-\tau})\\
		&\le\|u(\cdot,0)\|_{H^1}+|B|M_p.
	\end{aligned}
\end{equation}
Thus, \eqref{eq:uniform_control_H1_cont} holds for every $\xi_0\in\mathcal C$. 
Corollary \ref{cor:control_regularity_compact_cont}
now gives relative compactness and strong convergence for every $\xi_0\in\mathcal C$. 
Since $\mathcal C$ is dense in $\mathcal X$,
Corollary \ref{cor:dense_convergence_cont} gives
\eqref{eq:compact_convergence_cont} for every
$\xi_0\in\mathcal X$. 
\end{exmp}
The next example does not satisfy the (strong) accretivity assumption in \cite[Proposition 2.5(ii)]{Gernandt2025}; 
%{\rev (which, for $C^1$ functions, is equivalent to strong convexity)} 
we show that a global exponential estimate uniform over
initial values cannot be obtained. Nevertheless, the results in this paper apply and provide strong convergence for every mild solution, as shown in  Example \ref{ex:logcosh_convergence}.
\begin{exmp}
\label{ex:logcosh_nonexpconv}
Consider the system from Example \ref{ex:logcosh_convergence} with $A=B=1$, $T=1$ and $x_0=0$. Then $\bar x=0$, $\widetilde x=x$ and $h_{\bar x}=0$. 
% The optimal control problem is 
% \[
% \min \int_0^1
% \left(
% \log\cosh x(t)
% +
% \frac12 u(t)^2
% \right)\,dt,
% \qquad
% \dot x=x+u,
% \;
% x(0)=0.
% \]
The unique optimal state and control are zero 
%The costate equation with terminal condition in 
%
and \eqref{eq:PMP} gives $p_\star=0$. 
Then $\xi_\star=0$, and Example \ref{ex:logcosh_convergence} gives that 
% $\|S_I(\tau)\xi_0\|_{[L^2(0,1)]^3}\rightarrow 0$ for every $\xi_0\in [L^2(0,1)]^3$.  
every mild trajectory converges to $\xi_\star$. 
In the notation of \cite[Proposition 2.5(ii)]{Gernandt2025},
$\bar x_1=(x_\star,u_\star)=(0,0)$ and
$\mathcal M_1(x,u)=(\tanh x,u)$.
Strong accretivity at $\bar x_1$ would require, for
$x(t)=R$ and $u(t)=0$, that $R\tanh R\ge cR^2$
for some $c>0$ and all $R>0$.
This fails as $\tanh R/R\to 0$ as $R\to\infty$.

Suppose that for some $M\ge 1$ and $\omega>0$, independent of $\xi_0$, such that $\|S_I(\tau)\xi_0\|_{\mathcal X}\le M e^{-\omega\tau}\|\xi_0\|_{\mathcal X}$ for all $\tau\ge 0$ and $\xi_0\in \mathcal X$.
Define %the bounded operator 
$Q(\widetilde x,u,p):=(0,u,0)$, 
%and write $\partial\Phi(\xi)=Q\xi+F(\xi)$, 
$F(\xi):=(\tanh\widetilde x,0,0)$ and  
$G:=J-Q$, with $D(G)=D(J)$. Since $J$ is skew-adjoint and $Q$ is bounded,
self-adjoint, and nonnegative, then $G$ generates a linear contraction semigroup $T_G$ %in $[L^2(0,1)]^3$ (%Lumer-Phillips theorem, 
(\cite[Theorem 3.8.4]{TucsnakWeiss2009}). Equation \eqref{eq:weighted_optimizer_cont} becomes
\begin{equation}
\label{eq:Exmp4_nonlineq}
%\frac{d\xi}{d\tau}=G\xi-F(\xi).
\xi'=G\xi-F(\xi).
\end{equation} 
Since $\|F(\xi)\|_{\mathcal X}\le1$, variation of
constants gives 
$
%\begin{aligned}
\|S_I(\tau)(R\zeta)-R T_G(\tau)\zeta\|_{\mathcal \mathcal X}
% &\le
% \int_0^\tau
% \|T_G(\tau-s)\|\,
% \|F(S_I(s)(R\zeta))\|_{\mathcal X}\,ds \\
% &
\le\tau, \\R>0, \; \zeta\in\mathcal X.
%\int_0^\tau 1\,ds = \tau
%\end{aligned}
$
%\] 
%Note that the above bound does not depend on $R$.
Divide by $R$ and take $R\rightarrow\infty$
gives  
%\[
$\displaystyle
\|T_G(\tau)\|\leq Me^{-\omega\tau}$, $\tau\geq 0$. 
This estimate extends with the same constants to the
complexification $\mathcal X_{\mathbb C}$.
We keep the notation $G$ and $T_G$ for the
complexified operator and semigroup.  

For each $n\ge 1$, let
$k_n\in(n\pi,n\pi+\pi/2)$ solve 
$\tan k_n=k_n$, and denote $\beta_n=\sqrt{1+k_n^2}$. 
%(chosen to ensure the boundary conditions $\widetilde x_n(0)=0$). 
%Then $k_n>n\pi\to\infty$.  
Take
\begin{equation}
\label{eq:def_pn_xntilde}
p_n(t)=c_n\sin(k_n(1-t)),
\qquad
\widetilde x_n(t)
=\frac{-\dot p_n(t)-p_n(t)}{i\beta_n},
\end{equation}
where $c_n>0$ is chosen so that
$\|\widetilde x_n\|_{L^2(0,1)}^2
+\|p_n\|_{L^2(0,1)}^2=1$.
%Since $k_n\cos k_n=\sin k_n$, we have
Then $\widetilde x_n(0)=0$ and $p_n(1)=0$ 
%
%Moreover, 
and $\ddot p_n=-k_n^2p_n$ gives 
\begin{equation}
		\label{eq:pn_xntilde_identity}
	\begin{aligned}
		-\dot p_n-p_n
		=i\beta_n\,\widetilde x_n,\quad
		-\dot{\widetilde x}_n+\widetilde x_n
		=i\beta_n\,p_n.
	\end{aligned}
\end{equation}
Take 
\begin{equation}
	\label{eq:def_un_zetan}
u_n=-\frac{p_n}{1+i\beta_n},
\qquad
\zeta_n=(\widetilde x_n,u_n,p_n)\in D(G).
\end{equation}
Substitute $\zeta_n$ into $G$ and use
\eqref{eqn:J_Definition} to get  
% $G(\widetilde x,u,p)=
% \bigl(
% -\dot p-p,-u-p,\,-\dot{\widetilde x}+\widetilde x+u
% \bigr)$. 
% Therefore,
%\[
\[
\displaystyle
\bigl(G-i\beta_n\,I\bigr)\zeta_n
=
\begin{pmatrix}
	-\dot p_n-p_n-i\beta_n\,\widetilde x_n\\
	-(1+i\beta_n)u_n-p_n\\
	-\dot{\widetilde x}_n+\widetilde x_n+u_n
	-i\beta_n\,p_n
\end{pmatrix}.
\]
%\]
Use \eqref{eq:def_pn_xntilde},  \eqref{eq:def_un_zetan} and \eqref{eq:pn_xntilde_identity} to get 
$\bigl(G-i\beta_n\,I\bigr)\zeta_n
=(0,0,u_n)$ and
%\[
$\displaystyle 
\|u_n\|_{L^2(0,1)}
=\frac{\|p_n\|_{L^2(0,1)}}{\sqrt{2+k_n^2}}
\le\frac{1}{\sqrt{2+k_n^2}}
\rightarrow 0$ as $k_n\to\infty$. 
%\] 
%
% Since $\|\zeta_n\|_{\mathcal X_{\mathbb C}}\ge1$,
% the normalized vectors
% $\eta_n:=\zeta_n/\|\zeta_n\|_{\mathcal X_{\mathbb C}}$
% satisfy
% %\[
% $\|\eta_n\|_{\mathcal X_{\mathbb C}}=1$,
% %\qquad
% $\bigl\|\bigl(G-i\sqrt{1+k_n^2}\,I\bigr)\eta_n
% \bigr\|_{\mathcal X_{\mathbb C}}
% \rightarrow 0$.
% %\]
The exponential estimate for $T_G$ implies 
$i\beta\in\rho(G)$ and
%\[
$\displaystyle 
\|(i\beta I-G)^{-1}\|\le\frac{M}{\omega}$ for every 
$\beta\in\mathbb R$ (see e.g. \cite[Proposition 2.3.1 and Corollary 2.3.3]
{TucsnakWeiss2009}.).
%\]
Because $\eta_n:=\zeta_n/\|\zeta_n\|\in D(G)$, we write
%\[
$\eta_n=
\bigl(i\beta_n\,I-G\bigr)^{-1}
\bigl(i\beta_n\,I-G\bigr)\eta_n$. 
%\] 
This gives
$
%\begin{aligned}
	1=\|\eta_n\|_{\mathcal X_{\mathbb C}} \le\frac{M}{\omega}
	\bigl\|\bigl(i\beta_n\,I-G\bigr)\eta_n
	\bigr\|_{\mathcal X_{\mathbb C}}
	\\ \rightarrow0,
%\end{aligned}
$
%and this is 
a contradiction. Thus the nonlinear semigroup $S_I$ admits no global exponential estimate with constants independent of the initial value.
%, but every mild trajectory converges strongly to $\xi_\star=0$.
\end{exmp}

\section{Conclusions and Outlook}
As discussed in \cite{Gernandt2025, vanderSchaft2026PDG, vanderSchaft2026PDGJuly} the continuous time primal-dual gradient algorithm provides a promising starting point for deriving suboptimal controllers for the optimal control problem of linear systems with strictly convex costs. 
Importantly, the algorithm defines an infinite-dimensional incremental port-Hamiltonian system, which can be converted into standard port-Hamiltonian form by using a shifted Hamiltonian. In the present paper this has served as point of departure for a rigorous analysis of well-posedness and convergence of the algorithm, exploiting the port-Hamiltonian structure and ensuing dissipation inequality. Apart from a slightly different formulation, a major difference with regard to \cite{Gernandt2025} has been the assumption of \emph{strict} convexity of the running cost instead of \emph{strong} convexity. Furthermore, convergence has been proved using LaSalle's Invariance principle, thereby offering a rigorous alternative to the proof context of \cite{Gernandt2025}. 
We believe that the port-Hamiltonian structure of the primal-dual gradient algorithm can be exploited for obtaining efficient suboptimal controllers, by discretizing the port-Hamiltonian partial-differential equations in a structure-preserving way.

\vspace{-0.5cm}
\appendix
%\appendices

\section{Proofs}

\subsection{A proof of Proposition \ref{prn:-J_max_monotone}}
\label{App: proof_Proposition1}
The domain \eqref{eqn:J_Domain} contains all smooth
compactly supported triples and is dense in $\mathcal X$ since $C_c^\infty(0,T;\mathbb R^n)
\times
L^2(0,T;\mathbb R^m)
\times
C_c^\infty(0,T;\mathbb R^n)
\subset D(J)$.
For $\xi=(\widetilde x,u,p)$ and
$\eta=(\widetilde y,w,q)$ in $D(J)$, integration by parts
in \eqref{eqn:J_Definition} gives
$\langle J\xi,\eta\rangle_{\mathcal X}
	+\langle\xi,J\eta\rangle_{\mathcal X} =-[p^\top\widetilde y+\widetilde x^\top q]_0^T=0$.
Hence $J\subseteq -J^*$. 
Conversely, let $\eta=(\widetilde y,w,q)\in D(J^*)$.
Then
$\langle J\xi,\eta\rangle_{\mathcal X}
=\langle \xi,f\rangle_{\mathcal X}$,
$\xi\in D(J)$. 
Taking $\xi=(\varphi,0,0)$ and $\xi=(0,0,\varphi)$,
with $\varphi\in C_c^\infty(0,T;\mathbb R^n)$,
in the definition of $J^*$ and applying
\cite[Proposition 8.3, p. 206]{Brezis2011}
componentwise gives 
$\widetilde y,q\in H^1(0,T;\mathbb R^n)$. 
Taking $\xi=(0,v,0)$, with arbitrary
$v\in L^2(0,T;\mathbb R^m)$, gives $f_2=B^\top q$, and so we have 
\[
J^*\eta=
\begin{pmatrix}
	\dot q+A^\top q\\
	B^\top q\\
	\dot{\widetilde y}-A\widetilde y-Bw
\end{pmatrix}.
\]
Integration by parts for arbitrary $\xi\in D(J)$
then yields
$p(0)^\top\widetilde y(0)-\widetilde x(T)^\top q(T)=0$.
Since $p(0)$ and $\widetilde x(T)$ are arbitrary,
$\widetilde y(0)=q(T)=0$.
Thus $D(J^*)=D(J)$ and $J^*=-J$, so both $J$ and $-J$
are skew-adjoint. 
Then, $-J$ is monotone by skew-symmetry.
For every $\lambda>0$ and $\xi\in D(J)$,
$\|(\lambda I-J)\xi\|_{\mathcal X}^2
=\lambda^2\|\xi\|_{\mathcal X}^2+\|J\xi\|_{\mathcal X}^2$.
Since $J$ is closed, we have that
$R(\lambda I-J)$ is closed. Moreover,
$R(\lambda I-J)^\perp=\ker((\lambda I-J)^*)=\ker(\lambda I+J)=\{0\}$, where the last equality follows from $0=\langle(\lambda I+J)\eta,\eta\rangle_{\mathcal X}
=\lambda\|\eta\|_{\mathcal X}^2$ for $\eta\in\ker(\lambda I+J)$.
So $R(\lambda I-J)=\mathcal X$, and $-J$ is maximal monotone using Theorem \ref{thm: Theorem 1.141}. 

\vspace{-0.4cm}
\subsection{A proof of Proposition \ref{prn:Phi_max_monotone}}
\label{App: proof_Proposition2}

For every $t\in[0,T]$, the mapping
$(x,u)\mapsto K(x+\bar x(t),u)$ is convex, since $K$
is convex (see Assumption \ref{ass: K C1}) and the change of variables is affine \cite[Theorem 2.1.3(vi), p. 43]{Zalinescu2002}. 
Convexity of $\Phi$ then follows from \cite[p. 439]{Rockafellar1971}.
The required measurability and integrable lower bound
for each trajectory are verified in items (ii)-(iii) below. Moreover, $(0,0,0)\in D(J)\subseteq D(\Phi)$,
so $\Phi\not\equiv+\infty$. Since $\Phi$ never takes the
value $-\infty$, it is proper. 
Conditions (i)-(iv) in \cite[Proposition 2.53, p. 94]{BarbuPrecupanu2012} are satisfied with $g(t,x,u):= K(x+\bar x(t),u)$, $\alpha(t)=
\begin{pmatrix}
K_x\bigl(\bar{x}(t),0\bigr)\\
K_u\bigl(\bar{x}(t),0\bigr)
\end{pmatrix}$, 
$\beta(t)= K\bigl(\bar{x}(t),0\bigr)$ 
and  $p=2$. 
Let
$I_g:L^2(0,T;\mathbb R^{n+m})\to(-\infty,+\infty]$ be the integral functional
associated with $g$, using the same
integrability convention as $\Phi$.
By \cite[Proposition 2.53, p. 94]{BarbuPrecupanu2012},
$I_g$ is lower semicontinuous. The projection
$\pi:\mathcal X\to L^2(0,T;\mathbb R^{n+m})$,
$\pi(\widetilde x,u,p)=(\widetilde x,u)$,
is bounded and linear, and $\Phi=I_g\circ\pi$.
Therefore, $\Phi$ is lower semicontinuous (in the norm
topology of $\mathcal X$). Since $\Phi$ is proper convex and lower semicontinuous and $\mathcal X$ is a Hilbert space, Theorem \ref{thm: Theorem 2.43} implies that $\partial\Phi$ is maximal monotone on $X$. %This concludes the proof.

\vspace{-0.3cm}
\subsection{A proof of Proposition \ref{prn:equilibrium_cont}}
\label{App: proof_Proposition4}
By Assumption \ref{ass:existence_optimizer}, problem \eqref{eq:ocp_cost} admits an optimal pair
$(x_\star,u_\star)\in\mathcal F$.
Since $\mathcal F$ is affine and $K$ is strictly convex by Assumption \ref{ass: K C1}, this pair is unique, up to equality a.e., as already noted in Subsection \ref{sect_IPHsys}.

We apply the maximum principle in
\cite[Section 4.1.3, Theorem 4.5, pp. 243-244]
{BarbuPrecupanu2012}.
To fit the setting, take the state and control
spaces to be $\mathbb R^n$ and $\mathbb R^m$, respectively,
the integrability exponent to be $2$, the evolution operator
to be $e^{A(t-s)}$, and the control operator to be
$u\mapsto Bu$. Take zero forcing, no state constraint,
and the running cost $L(t,x,u)=K(x,u)$.
The fixed initial and the free terminal state are
encoded by the endpoint cost $\ell(x_0,b)=0$ and $\ell(a,b)=\infty$ whenever $a \neq x_0$.
% 	0,       & a=x_0,\\
% 	+\infty, & a\neq x_0
% %\[
% $\displaystyle 
% \ell(a,b):=
% \begin{cases}
% 	0,       & a=x_0,\\
% 	+\infty, & a\neq x_0.
% \end{cases}$
% %\]
For these finite-dimensional dynamics, a continuous mild
state associated with $u\in\mathcal U$ belongs to
$H^1(0,T;\mathbb R^n)$, since its weak derivative is
$Ax+Bu\in L^2(0,T;\mathbb R^n)$.
Consequently, the finite-cost admissible pairs in the 
formulation in \cite[Section 4.1.3, Theorem 4.5, pp. 243-244]
{BarbuPrecupanu2012} coincide with $\mathcal F$. It can be directly checked that the hypotheses (A)-(E) of this theorem, listed in \cite[pp. 234-236]{BarbuPrecupanu2012}, are satisfied. 
%
% {\rev Verify (A)-(E) are SATISFIED - 
% IN Appendix \ref{proofprprn:equilibrium_cont} - TO BE ADDED}
%
Theorem 4.5 therefore characterizes optimal pairs by the
extremality conditions of
\cite[Definition 4.4, equations (4.17)-(4.21),
pp. 242-243]{BarbuPrecupanu2012}.
Denote the dual arc used in \cite{BarbuPrecupanu2012} by $\widehat p$.
Since there is no state constraint, the normal-cone
measure in those conditions vanishes.
Moreover,
%\[
$\partial\ell(x_0,x_\star(T))
=\mathbb R^n\times\{0\}$, 
%\]
so the endpoint condition gives $\widehat p(T)=0$.
Since $K$ is continuously differentiable, the 
subgradient condition and adjoint representation in \cite{BarbuPrecupanu2012} reduce to
%\[
$B^\top\widehat p(t)
=K_u(x_\star(t),u_\star(t))$
a.e. on $(0,T)$, 
%\]
and
%\[
$\displaystyle
\widehat p(t)
=-\int_t^T e^{A^\top(s-t)}
K_x(x_\star(s),u_\star(s))\,ds$,
$t\in[0,T]$.
%\]
Their dual arc $\widehat p$ has the opposite sign to the costate $p_\star$ in \eqref{eq:PMP}. Setting $p_\star=-\widehat p$ therefore
gives \eqref{eq:costate_formula_cont} and the stationarity
condition in \eqref{eq:PMP}. 

From Assumption \ref{ass: K Global Lipschitz} and
$(x_\star,u_\star)\in
L^2(0,T;\mathbb R^n)\times L^2(0,T;\mathbb R^m)$,
we have
$K_x(x_\star,u_\star)\in L^2(0,T;\mathbb R^n)$. 
The integral representation
\eqref{eq:costate_formula_cont} consequently defines a
continuous, absolutely continuous costate that satisfies
%\[
$\dot p_\star
=-A^\top p_\star-K_x(x_\star,u_\star)$ 
a.e. on $(0,T)$, $p_\star(T)=0$.
%\]
Its derivative belongs to $L^2(0,T;\mathbb R^n)$, so
$p_\star\in H^1(0,T;\mathbb R^n)$.
Thus all equations in \eqref{eq:PMP} hold, and
$\xi_\star=(x_\star-\bar x,u_\star,p_\star)\in D(J)$
by \eqref{eqn:J_Domain} and $\bar x(0)=x_0$. 
Using \eqref{eqn:J_Definition},
\eqref{eq:subgradient_singleton_cont}, and
\eqref{eq:lift_forcing_cont}, the optimality system
\eqref{eq:PMP} gives
$J\xi_\star-\partial\Phi(\xi_\star)+h_{\bar x}=0$.
Hence $\xi_\star$ is an equilibrium of
\eqref{eq:weighted_optimizer_cont}.

Conversely, let $\xi=(\widetilde x,u,p)\in D(J)$ be an equilibrium and put $x=\widetilde x+\bar x$.
The same identities, together with the boundary conditions in \eqref{eqn:J_Domain}, imply \eqref{eq:PMP} with
$(x,u,p)$ in place of $(x_\star,u_\star,p_\star)$.
With $\widehat p=-p$, $q=K_x(x,u)$, $w\equiv0$,
and $\widehat p_T=0$, the pair $(x,u)$ together with these auxiliary variables satisfies the %extremality 
conditions in \cite[Definition 4.4, (4.17)-(4.21), pp. 242-243]{BarbuPrecupanu2012}. 
Using \cite[Theorem 4.5]{BarbuPrecupanu2012}, every extremal pair is optimal, therefore $(x,u)$ is optimal. 
Uniqueness of the optimal pair gives
$(x,u)=(x_\star,u_\star)$. 
The terminal-value adjoint equation has the unique solution \eqref{eq:costate_formula_cont}, so $p=p_\star$ and $\xi=\xi_\star$. 
Finally, \eqref{eq:weighted_operator_cont} and the
invertibility of $W$ give
$\{\xi\in D(\mathcal A_W):\mathcal A_W\xi=0\}
=\{\xi_\star\}$.
\end{document}